\documentclass[12pt, letterpaper]{article}
\usepackage{amsmath,amsthm,amssymb,latexsym,amscd}
\usepackage{graphicx}
\usepackage{enumerate}
\usepackage{times,helvet} 
\usepackage{mathrsfs} 
\usepackage{geometry}
\usepackage[bookmarks=true]{hyperref}\hypersetup{pdfauthor = {Ira Gessel, Ji Li},pdftitle = {Enumeration of Point-Determining Graphs},pdfkeywords = {graph, combinatorics, combinatorial species theory, point-determining, bicolored graphs,...},pdfcreator = {LaTeX with hyperref package},pdfproducer = {dvips + ps2pdf}}
\usepackage{breakurl}
\def\mytopsep{3mm}
\newtheoremstyle{myplain}{\mytopsep}{\mytopsep}{\itshape}{0pt}{\bfseries}{.}{3mm}{}
\newtheoremstyle{mydefinition}{\mytopsep}{\mytopsep}{\normalfont}{0pt}{\bfseries}{.}{3mm}{}
\newtheoremstyle{myremark}{\mytopsep}{\mytopsep}{\normalfont}{0pt}{\bfseries}{.}{3mm}{}
\theoremstyle{myplain}
\newtheorem{thm}{Theorem}[section]
\newtheorem{cor}[thm]{Corollary}
\newtheorem{lem}[thm]{Lemma}
\newtheorem{prop}[thm]{Proposition}

\theoremstyle{mydefinition}
\newtheorem{dfn}[thm]{Definition}
\theoremstyle{myremark}

\numberwithin{equation}{section}
\newcommand\mc{\mathscr}
\newcommand\wt{\widetilde}
\newcommand\inverse{\langle-1\rangle}
\newcommand{\aatop}[2]{\genfrac{}{}{0pt}{}{#1}{#2}} 
\newcommand\edge[2]{\{#1,#2\}} 
\DeclareMathOperator{\fix}{fix}
\DeclareMathOperator{\lcm}{lcm} 
\begin{document}
\title{Enumeration of Point-Determining Graphs}
\author{Ira M. Gessel\footnote{{\it Email address}: gessel@brandeis.edu}\\{\small\it Department of Mathematics, Brandeis University, MS 050, Waltham, MA 02454-9110}\\[5pt]
Ji Li\footnote{{\it Email address}: jli@math.arizona.edu}\thanks{This work is supported in part by the National Science Foundation Grant Number  DUE-0634532.}\\{\small\it Department of Mathematics, The University of Arizona, 617 N. Santa Rita Ave., Tucson, AZ 85721-0089}}
\date{\small \today}
\maketitle
\begin{abstract}
Point-determining graphs are graphs in which no two vertices have the same neighborhoods, co-point-determining graphs are those whose complements are point-determining, and bi-point-determining graphs are those both point-determining and co-point-determining. Bicolored point-determining graphs are point-determining graphs whose vertices are properly colored with white and black. We use the combinatorial theory of species to enumerate these graphs as well as the connected cases.
\end{abstract}

\setcounter{section}{-1}

\section{Introduction}

A \emph{point-determining graph} (also called a \emph{mating-type graph},  \emph{mating graph}, or \emph{M-graph}) is a graph in which no two vertices have the same neighborhood. The term ``point-determining" was introduced by Sumner~\cite{sumner}. If we start with any graph, and identify vertices with the same neighborhood, we obtain a point-determining graph. This was the motivation for their use by Bull and Pease~\cite{bull} to represent mating systems: Let each vertex of a graph denote an individual animal and let two vertices be joined if the two animals can mate. If two animals have identical compatibilities they are said to be of the same mating type, and in that case they don't need to be represented by different vertices. Thus we may represent animals with the same mating type by the same vertex, resulting in a graph in which no two vertices have identical neighborhoods. 

Point-determining graphs (both labeled and unlabeled) were counted by Read~\cite{read1}, 
using this reduction of arbitrary graphs to point-determining graphs, and we shall follow the same approach in this paper.

Complements of point-determining graphs, which we call co-point-determining graphs (they have also been called ``point-distinguishing"), are graphs in which no two vertices have the same closed neighborhood. (The closed neighborhood of a vertex is the vertex together with its neighborhood.) The enumeration of co-point-determining graphs is of course the same as the enumeration of point-determining graphs, but counting connected point-determining and co-point-determining graphs is different. Surprisingly, the number of unlabeled point-determining and co-point-determining graphs is the same, but for the labeled versions, the number of point-determining and co-point-determining graphs on $n$ vertices differ by $(n-1)!$.

Next we count graphs which are both point-determining and co-point-determining, which we call bi-point-determining. (They have also been called ``totally point-determining".) These graphs may also be characterized by the property that their automorphism groups contain no transpositions; i.e., they are not fixed by switching any pair of vertices. Just as an arbitrary graph can be reduced to a point-determining graph by identifying graphs with the same neighborhood, an arbitrary graph may be reduced to a bi-point-determining graph by a more complicated compression in which the fibers are cographs (graphs obtained from edgeless graphs by complementation and union).

We use the combinatorial theory of species~\cite{joyal, joyal2, species-book} as our framework for graphical enumeration. In section~\ref{sec-species}, we introduce some terminology and basic results of species theory. The superimposition of graphs is defined and related to composition of species  in Lemma~\ref{lem-super}. In section~\ref{sec-pd}, we enumerate point-determining graphs through a functional relation between the species of point-determining graphs and the well-known species
of graphs (Theorem~\ref{thm-pdcpd}), and examine the species of connected point-determining graphs and connected co-point-determining graphs (Theorem~\ref{thm-connpdcpd}). In section~\ref{sec-noend}, we describe connection between unlabeled connected point-determining graphs and unlabeled connected graphs without endpoints (Corollary~\ref{cor-noendpd}), which was previously studied in~\cite{goulden}, ~\cite{wright}, and~\cite{kilibarda}. In section~\ref{sec-bipd}, we find a functional relation between the species of bi-point-determining graphs and the species of graphs (Theorem~\ref{thm-bipd}). The enumeration of connected bi-point-determining graphs is carried out using virtual species (Corollary~\ref{cor-connbipd}). In section~\ref{sec-2color}, we examine the $2$-sort species of \emph{bicolored graphs} (Theorem~\ref{thm-2color}), which are graphs whose vertices are properly colored with white and black, and develop ways to enumerate the bicolored point-determining graphs and the connected ones (Theorem~\ref{thm-2colorpd},~\ref{thm-2colorps}).

A list of species covered in this paper is given in appendix~\ref{appen-index}. In appendix~\ref{appen-species} we list some computational results on the cycle indices and molecular decompositions of species.

\section{Combinatorial Species and Superimposition of Graphs}
\label{sec-species}

The combinatorial theory of species was initiated by Joyal in~\cite{joyal} and ~\cite{joyal2}. For detailed definitions and descriptions about species, readers are referred to~\cite{species-book}. A \emph{species} is a functor from the category of finite sets with bijections to itself. A species $F$ generates for each finite set $U$ a finite set $F[U]$, called the set of \emph{$F$-structures} on $U$, and for each bijection $\sigma: U \to V$ a bijection $F[\sigma]: F[U]\to F[V],$ called the \emph{transport of $F$-structures along $\sigma$}. The symmetric group $\mathfrak{S}_n$ acts on the set $F[n]=F[\{ 1,2,\dots,n\}]$ by transport of structures. The $\mathfrak{S}_n$-orbits under this action are called \emph{unlabeled} $F$-structures of order $n$.

Each species $F$ is associated with three generating series: the \emph{exponential generating series} $F(x)=\sum_{n\ge 0} |F[n]|{x^n}/{n!},$ the \emph{type generating series} $\wt{F}(x)=\sum_{n\geq 0} f_n\, x^n,$ where $f_n$ is the number of unlabeled $F$-structures of order $n$, and the \emph{cycle index} 
$$Z_F=Z_F\,(p_1,p_2,\dots)=\sum_{n\geq 0} \biggl(\sum_{\lambda \vdash n} \fix \, F[\lambda]\,\frac{p_{\lambda}} {z_{\lambda}} \biggr),$$
where $\fix \,F[\lambda]$ denotes the number of $F$-structures on $[n]=\{1,2,\dots,n\}$ fixed by $F[\sigma]$ for some $\sigma$ that is a permutation of $[n]$ with cycle type $\lambda$, $z_{\lambda}$ is the number of permutations in  $\mathfrak{S}_n$ that commute with a permutation of cycle type $\lambda$, and $p_{\lambda}$ is the \emph{power sum symmetric function} (see~\cite[p.~297]{ec2}) indexed by the partition $\lambda$ of $n$, defined by
\begin{align*}
p_n &= p_n[{\bf x}]=\sum_{i}x_i^n, \quad n \ge 1\\
p_{\lambda}&= p_{\lambda}[{\bf x}]=p_{\lambda_1}p_{\lambda_2}\dots=\prod_{k \ge 1} p_k^{c_k(\lambda)}, \text{ if }
\lambda=(1^{c_1(\lambda)}, 2^{c_2(\lambda)}, \dots),
\end{align*}
where $c_i(\lambda)$ is the number of parts of length $i$ in the partition $\lambda$, and hence $\sum_i ic_i(\lambda)=n$. The $p_\lambda$ form a basis for the ring of symmetric functions in the variables $x_1, x_2, \dots$. We can also define power sum symmetric functions in the variables $y_1, y_2, y_3, \dots$, written as $p_{\lambda}[{\bf y}]$, in a similar fashion.
 
The following identities \cite[p.~18]{species-book} illustrate the importance of the cycle index in the theory of species.
\begin{align}
F(x)&=Z_F(x, 0, 0, \dots), \notag \\
\wt{F}(x)&=Z_F(x, x^2, x^3, \dots). \notag
\end{align}

We may apply operations on species to generate new species, and the operations of species translate into operations of the generating series of species systematically. For details about operations of species, readers are referred to \cite[pp.~1--58]{species-book}. The species operations that are frequently used in this paper are the \emph{sum} $\Phi+\Psi$, the \emph{product} $\Phi  \Psi$ or $\Phi\cdot \Psi$, and the \emph{composition} $\Phi(\Psi)$ or $\Phi\circ \Psi$ of species $\Phi$ and $\Psi$. We recall here the definition of the composition of $\Phi$ and $\Psi$ \cite[p.~41]{species-book}: A $\Phi\circ\Psi$-structure on a finite set $U$ is a triple of the form $(\pi, f, \gamma)$, where $\pi$ is a partition of $U$, $f$ is a $\Phi$-structure on the set of blocks of $\pi$, and $\gamma=(\gamma_B)_{B\in\pi}$, where for each block $B$ of $\pi$, $\gamma_B$ is a $\Psi$-structure on $B$. The formulas for the associated series of $\Phi\circ\Psi$ are given by
\begin{align*}
(\Phi\circ \Psi)(x)&=\Phi(\Psi(x)),\\
(\wt{\Phi\circ \Psi})(x)&=Z_{\Phi}(\wt{\Psi}(x),\wt{\Psi}(x^2),\dots),\\
Z_{\Phi\circ \Psi}&=Z_{\Phi}\circ Z_{\Psi},
\end{align*}
where $\circ$ is the operation of \emph{plethysm} on symmetric functions (see~\cite[p.~447]{ec2}).

If $F$ is a species of structures, we denote by $F_n$, for a nonnegative integer $n$, the species of $F$-structures \emph{concentrated on the cardinality $n$} (see~\cite[p.~30]{species-book}), and by $F_{\ge n}$ the $F$-structures of cardinality at least $n$. Hence $F_{\ge n}=F_n+F_{n+1}+\cdots.$ We usually write $F_{\ge 1}$ as $F_+$.

A species $M$ is called a \emph{molecular species} (see~\cite{yeh1} and~\cite{yeh2}) if there is only one isomorphism class of $M$-structures. Thus a molecular species is one that is indecomposable under addition. Every species can be expressed uniquely as the sum of molecular species, and this expression is called its \emph{molecular decomposition} (see~\cite[p.~141]{species-book}).
 
Throughout this paper, we consider only simple graphs (without loops or multiple edges). A graph $G$ is thought of as an ordered pair $(V,E)$, where $V=V(G)$ is the vertex set of $G$, and $E=E(G)$ is the edge set of $G$, a set of 2-subsets of $V$. Two graphs are called \emph{disjoint} if they have no common vertices. An \emph{unlabeled graph} is formally defined as an isomorphism class of graphs, though we think of an unlabeled graph as simply a graph without vertex labels. A graph with no vertices is called \emph{empty}. The empty graph is not considered to be a connected graph. The \emph{empty} species, denoted by $0$, is defined by $0[U]=\varnothing$ for all $U$. The species of the empty graph is $1$. The species of the singleton graph is denoted by $X$. We denote by $\mc{E}$ the species of \emph{sets}. A fundamental property of $\mc{E}$ that we shall use several times is that for any species (or virtual species) $F$ and $G$, $\mc{E}(F+G) = \mc{E}(F)\mc{E}(G)$.

The cycle index of the species $\mc{G}$ of graphs was given in~\cite[p.~79]{species-book} and~\cite[p.~334, Theorem 2]{Robinson-nonseparable}:$$Z_{\mc{G}}=\sum_{n\ge 0} \biggl( \sum_{\lambda \vdash n} \fix\,\mc{G}[\lambda] \,\frac{p_{\lambda}}{ z_{\lambda}}\biggr),$$
where $$ \fix\,\mc{G}[\lambda]=2^{\frac{1}{ 2}\sum_{i,j \ge 1}\gcd(i,\,j)\,c_i(\lambda)c_j(\lambda) - \frac{1}{ 2}\sum_{k\ge1}(k \bmod 2)\,c_k(\lambda)}.$$

A \emph{virtual species} is a formal difference of species (see~\cite[p.~121]{species-book}). Since there is only one $(1+X)$-structure on the empty set, the species $1+X$ satisfies $(1+X)(0)=1$. Proposition $18$ of~\cite[p.~129]{species-book} asserts that there exists a unique virtual species which we denote by  $(1+X)^c$, the virtual species of  ``connected $(1+X)$-structures" with $1+X=\mc{E}\circ(1+X)^c$, or equivalently, $X=\mc{E}_+\circ(1+X)^c.$ Thus $(1+X)^c$ is referred to as the ``combinatorial logarithm of the species $1+X$'' (see~\cite[p.~131]{species-book}), or the compositional inverse of $\mc{E}_+$. 

\begin{lem}
The associated series of $(1+X)^c$ are 
\begin{align*}
(1+X)^c(x) &= \log (1+x), \\
\wt{(1+X)^c}(x)&= x-x^2\\
Z_{(1+X)^c}&= 
 \sum_{k\ge 1}\frac{\mu(k)}{k} \log (1+p_k),
\end{align*}
where $\mu$ denotes the \emph{M\"obius function}, defined by
\[\mu(k)=\left\{
\begin{array}{ll}
0, & \text{ if } n \text{ has one or more repeated prime factors},\\
1, & \text{ if } n=1,\\
(-1)^j, & \text{ if } n \text{ is a product of } j\ \text{distinct
primes.}
\end{array} \right.\]
\end{lem}

\begin{proof} The first and third formulas are special cases of general formulas for combinatorial logarithms given 
in \cite[p.~131]{species-book}. For the second formula, we have from \cite[p.~131, eq. (58)]{species-book},
\begin{equation}
\label{eqn-read}
\wt{(1+X)^c}(x)= \sum_{k\ge 1} \frac{\mu(k)}{k} \log
\wt{(1+X)}(x^k)= \sum_{k\ge 1} \frac{\mu(k)}{k} \log (1+x^k).
\end{equation}
The coefficient of $x^n$ in \eqref{eqn-read} is easily seen to be $\frac 1n\sum_{d|n}\mu(d) (-1)^{n/d -1}$.
Now define $g$ by $g(1)=1$, $g(2)=-2$, and $g(n) = 0$ for $n>2$. Then $\sum_{d|n}g(d) = (-1)^{d-1}$, so by M\"obius inversion, $\sum_{d|n}\mu(d) (-1)^{n/d -1} = g(n)$.

A different proof of the second formula was given by Read \cite{read1}.
\end{proof}

We denote by $\mc{K}$ the species of \emph{complete graphs}, which are graphs in which each pair of vertices are adjacent. The complement of a complete graph is called an \emph{edgeless graph}. The species of edgeless graphs, which are graphs with isolated vertices, may be identified with the species $\mc{E}$ of sets. We see that there is a natural transformation $\alpha$ that produces for every finite set $U$ a bijection between $\mc{E}[U]$ and $\mc{K}[U]$, namely, sending the edgeless graph on $U$ to the complete graph with vertex set $U$. Note that this bijection is carried through the complementation of graphs. The following diagram commutes for any finite sets $U$, $V$ and any bijection $\sigma: U \to V$:
\begin{center}
 $\begin{CD}
 \mc{E}[U] @>\mc{E}[\sigma]>> \mc{E}[V] \\
 @V{\alpha}VV @VV{\alpha}V \\
 \mc{K}[U] @>\mc{K}[\sigma]>> \mc{K}[V] \\
\end{CD}$
\end{center}
In this case we call these two species \emph{isomorphic} to each other (see~\cite[p.~21]{species-book} for the general definition of two species being isomorphic). Two isomorphic species essentially possess the ``same'' combinatorial properties. Thus we write $\Phi=\Psi$ to mean the species $\Phi$ is isomorphic to the species $\Psi$, and say there is a \emph{combinatorial equality} (see~\cite[p.~21]{species-book}) between them.

The theory of \emph{multisort species} (see~\cite[p.~100]{species-book}) is analogous to multivariate functions. A $2$-sort species $F(X,Y)$ generates for each finite two-set $U=(U_1, U_2)$ a finite set $F[U_1, U_2]$, where elements in  $F[U_1, U_2]$ are called $F$-structures on $U$. Furthermore, for any multibijection $$\sigma=(\tau_1, \tau_2): (U_1, U_2)\rightarrow(V_1, V_2),$$ where $\tau_1$ is a bijection from $U_1$ to $V_2$ and $\tau_2$ is a bijection from $U_2$ to $V_2$,  the transport of $F(X,Y)$-structures along $\sigma$ is  $$F[\sigma]=F[\tau_1,\tau_2]: F[U_1, U_2] \rightarrow F[V_1,V_2].$$ Moreover, the functions $F[\sigma]$ must satisfy the functoriality properties. In this paper, we consider the 2-sort species $\mc{H}(X,Y)$ of connected graphs in which vertices of degree one have sort $Y$ and all other vertices have sort $X$ (in section~\ref{sec-noend}), and $\mc{G}(X,Y)$ of bicolored graphs in which white vertices are of sort $X$ and black vertices are of sort $Y$ (in section~\ref{sec-2color}).

We introduce a kind of decomposition for graphs that will be helpful in counting point-determining and especially bi-point-determining graphs. 
\begin{dfn}
\label{dfn-superimposition}
Let $H_1,\dots, H_m$ be graphs with disjoint vertex sets, and let $G$ be a graph with vertex set $\{V(H_1), \dots, V(H_m)\}$.  We define the {\it superimposition} $G|_{H_1,\dots,H_m}$ of $G$ on $\{H_1, \dots, H_m\}$ to be the graph with vertex set $\bigcup_{i=1}^m V(H_i)$ in which $\edge uv$ is an edge if it is an edge of some $H_i$ or if $u\in V(H_i)$ and $v\in V(H_j)$ for some $i\ne j$, and $\edge{V(H_i)}{V(H_j)}\in E(G)$. 
\end{dfn}
Figure~\ref{f-superimp_exa} illustrates the superimposition of a graph $G$ on a set of graphs $\{H_1, H_2, H_3\}$.
\begin{figure}[ht]
\begin{center}
\includegraphics[width=12cm]{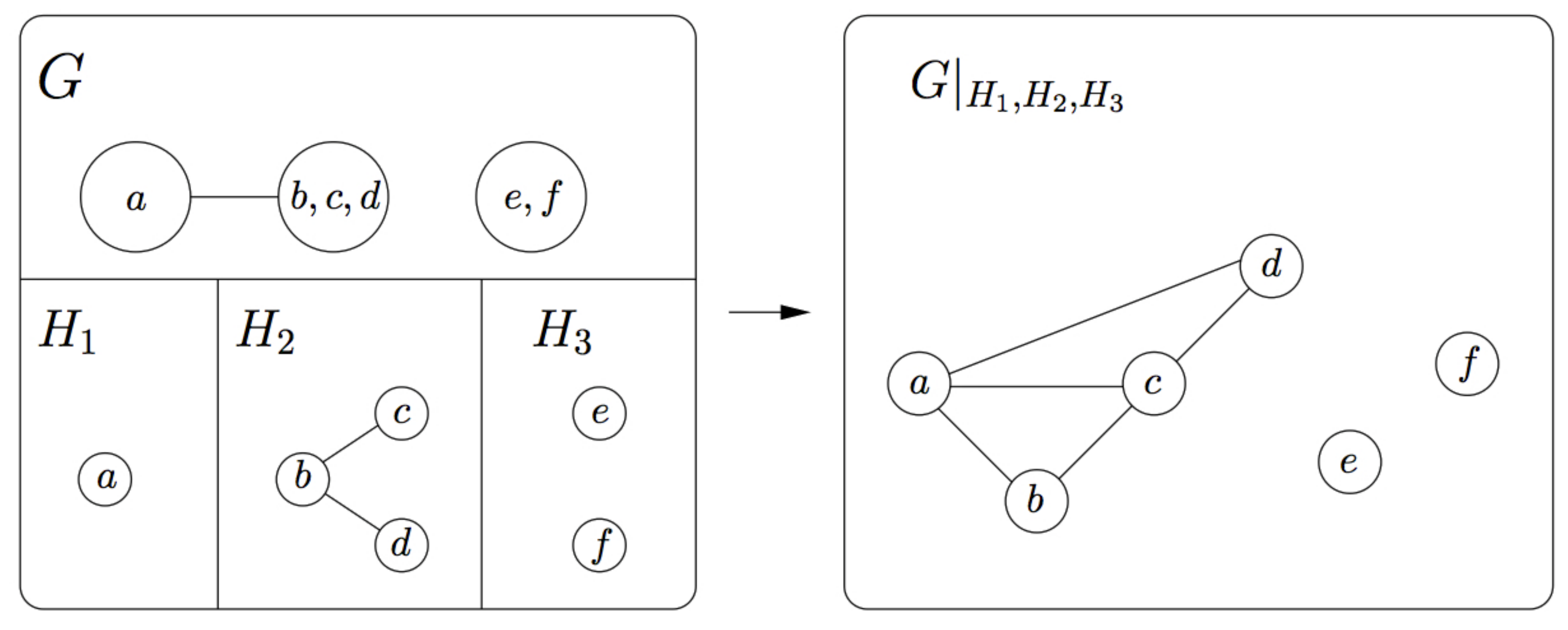}
\end{center}
\caption{\label{f-superimp_exa}
The superimposition $G|_{H_1, \,H_2,\, H_3}$.}
\end{figure}

We introduce two special cases of superimposition. Let $n$ be any positive integer. The {\it edgeless graph} of order $n$ is a graph with $n$ isolated vertices, denoted $E_n$. The {\it complete graph} of order $n$ is a graph with $n$ vertices each pair of which is adjacent to each other, denoted $K_n$. 

\begin{dfn}
\label{dfn-unionjoin}
Let $\{G_1, \dots, G_n\}$ be a set of nonempty pairwise disjoint graphs. 
We define the {\it union} of $\{G_1,\dots, G_n\}$ to be the superimposition $E_n|_{G_1,\dots,G_n}$, and the {\it join} of $\{G_1, \dots, G_n\}$ to be the superimposition $K_n|_{G_1,\dots,G_n}$, where  the vertex set of $E_n$ and $K_n$ is $\{V(G_1),\dots, V(G_n)\}$.
\end{dfn}

The operation of superimposition of species of graphs is closely related to composition of species. Let $\Phi$ and $\Psi$ be two species of graphs; i.e., for every finite set $U$, $\Phi[U]$ and $\Psi[U]$ are sets of graphs with vertex set $U$.
 We define a species $\Phi\diamond \Psi$ for which $(\Phi\diamond \Psi)[U]$ is the set of all superimpositions $G|_{H_1,\dots,H_m}$ in which $H_1,\dots, H_m$ are $\Psi$-graphs with $\bigcup_{i=1}^m V(H_i) = U$ and $G$ is a $\Phi$-graph with vertex set $\{V(H_1),\dots, V(H_n)\}$. 

It is clear from the definitions that there is a species map (see Definition 12, \cite[p.~21]{species-book}), from $\Phi\circ \Psi$ to $\Phi\diamond \Psi$. The following lemma, whose proof is straightforward, will be essential in our enumerative applications of superimposition.

\begin{lem} 
\label{lem-super}
Let $\Phi$ and $\Psi$ be species of graphs such that every $\Phi\diamond \Psi$-graph can be expressed uniquely as a superimposition of a $\Phi$-graph on a set of $\Psi$-graphs. Then $\Phi\circ \Psi$ is isomorphic to $\Phi\diamond\Psi$. \qed
\end{lem}

Note that the definition of superimposition of species of graphs is not a species operation in the sense that isomorphic species are not equivalent with respect to superimposition. For example, the species $\mc{E}_+$ of nonempty edgeless graphs is isomorphic to the species $\mc{K}_+$ of nonempty complete graphs, but $\mc{E}\diamond \mc{K}_+$ is the species of graphs all of whose connected components are complete and $\mc{E}\diamond \mc{E}_+$ is $\mc{E}$, the species of edgeless graphs. 

It is clear that superimposition is compatible with complementation of graphs, and that the following lemma holds.

\begin{lem}
\label{lem-supercomposition}
Suppose that the species of graphs $\Phi$ and $\Psi$ satisfy the hypotheses of Lemma \ref{lem-super}. Let $\bar \Phi$ be the species of complements of $\Phi$-graphs and let $\bar\Psi$ be the species of complements of $\Psi$-graphs. 
Then $\bar\Phi\circ \bar\Psi$ is isomorphic to $\bar\Phi\diamond\bar\Psi$, which is the species of complements of $\Phi\diamond\Phi$-graphs. \qed
\end{lem}

\section{Point-Determining Graphs}
\label{sec-pd}

Let $v$ be a vertex of a graph $G$. The \emph{neighborhood} $N(v)$ of $v$ in $G$ is the set of vertices adjacent to $v$. That is, $N(v)=\{ w\in V(G): \edge{v}{w}\in E(G)\}$. The \emph{closed neighborhood} $\bar{N}(v)$ of $v$ is $\bar{N}(v)=N(v)\cup \{v\}$. Note that if $N(v)=N(w)$, then $v$ is not adjacent to $w$, and if $\bar{N}(v)=\bar{N}(w)$, then $v$ is adjacent to $w$. In~\cite{corneil}, vertices with the same neighborhoods are called \emph{weak siblings}, and vertices with the same closed neighborhoods are called \emph{strong siblings}.

\begin{dfn}\label{dfn-pdcpd}
A \emph{point-determining graph} is a graph $G$ in which distinct vertices have distinct neighborhoods. A graph is called \emph{co-point-determining} if its complement is a point-determining graph. 
\end{dfn}

Note that a graph is co-point-determining if and only if distinct vertices  have distinct closed neighborhoods. Since the neighborhood of an isolated vertex (a vertex of degree $0$) is the empty set,  a point-determining graph has at most one isolated vertex. We regard the empty graph as both point-determining and co-point-determining. 

Since the complement of a point-determining graph is co-point-determining, the species $\mc{P}$ of point-determining graphs and the species $\mc{Q}$ of co-point-determining graphs are isomorphic, written as $\mc{P}=\mc{Q}$. Note that, as in this case, two species of graphs can be isomorphic without the corresponding graphs being isomorphic.

\begin{thm}
\label{thm-pdcpd}
For the species $\mc{G}$ of graphs, the species $\mc{P}$ of point-determining graphs, the species $\mc{E}_+$ of nonempty edgeless graphs, the species $\mc{Q}$ of co-point-determining graphs, and the species $\mc{K}_+$ of nonempty complete graphs, we have
\begin{equation}\label{eqn-pdcpd}
\mc{G}=\mc{P}\circ\mc{E}_+=\mc{Q}\circ\mc{K}_+.
\end{equation}
\end{thm}

\begin{proof}
We show that every graph can be expressed uniquely as a superimposition of a point-determining graph on a set of edgeless graphs. We first prove uniqueness. Suppose that $G$ is a superimposition $P|_{H_1,\dots, H_m}$ where $P$ is point-de\-ter\-mining and each $H_i$ is edgeless. (See Figure~\ref{f-graph_pd} for an example of this construction.) We define an equivalence relation on $V(G)$ in which vertices $u$ and $v$ are equivalent if they have the same neighborhood. Then since $H_i$ is edgeless, any two vertices of $V(H_i)$ must have the same neighborhood in $G$, so $V(H_i)$ is contained in an equivalence class. If $V(H_i)$ and $V(H_j)$ were contained in the same equivalence class, where $i\ne j$, then $V(H_i)$ and $V(H_j)$ would have the same neighborhood in $P$, so $P$ would not be point-determining. Therefore the vertex sets $V(H_i)$ must be the equivalence classes. It is easily seen that $P$ must be the graph on the equivalence classes $V(H_1)$, \dots, $V(H_m)$ in which there is an edge from $V(H_i)$ to $V(H_j)$ if and only if there is an edge of $G$ from each element of $V(H_i)$ to each element of $V(H_j)$. 
Conversely, it is easily seen that this construction does indeed express $G$ as a superimposition of a point-determining graph on a set of edgeless graphs.

The second equality follows from Lemma~\ref{lem-supercomposition}.
\begin{figure}[htb]
\begin{center}
\includegraphics[width=12cm]{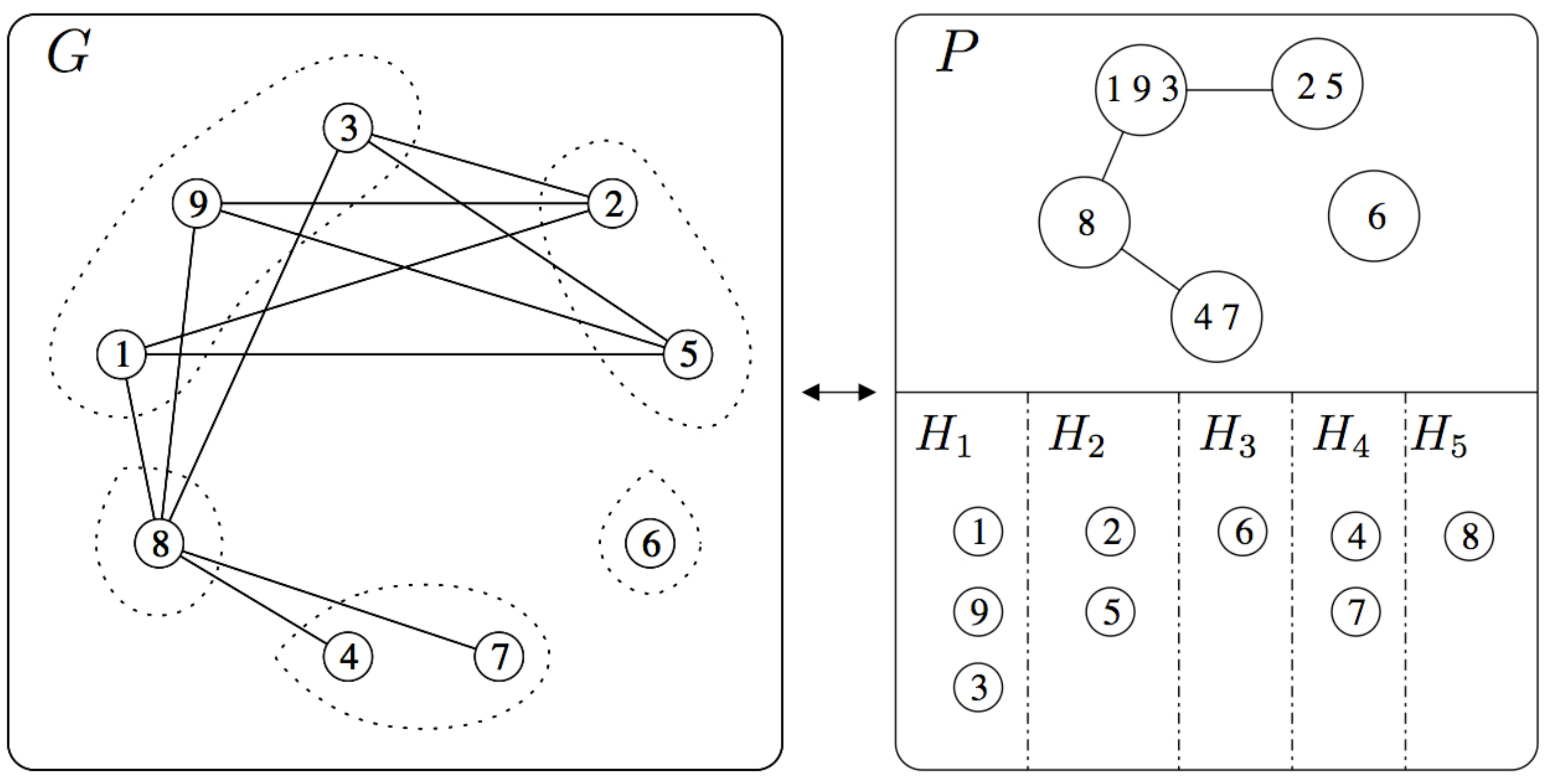}
\end{center}
\caption
{\label{f-graph_pd} $G = P|_{H_1,\dots, H_m}$}
\end{figure}
\end{proof}

Recall that $(1+X)^c$  is the compositional inverse of $\mc{E}_+$.  It follows from~\eqref{eqn-pdcpd} that 
\begin{equation}\label{eqn-qgamma}
\mc{P}=\mc{Q}=\mc{G}\circ(1+X)^c,
\end{equation}
which gives rise to several identities that can be used to compute the associated series of $\mc{P}$:
\begin{align}
\mc{P}(x)&=\mc{Q}(x)=\mc{G}(\log(1+x)), \label{eqn-pdcpd-label}\\
\wt{\mc{P}}(x)&=\wt{\mc{Q}}(x)=Z_{\mc{G}}( x-x^2, x^2-x^4, \dots ), \label{eqn-pdcpd-unlabel}\\
Z_{\mc{P}}&=Z_{\mc{Q}}=Z_{\mc{G}}\biggl( \sum_{k\ge
1}\frac{\mu(k)}{k} \log(1+p_k), \sum_{k\ge 1}\frac{\mu(k)}{k} \log
(1+p_{2k}), \ \dots \biggr).\notag
\end{align}

Read derived formulas~\eqref{eqn-pdcpd-label} and~\eqref{eqn-pdcpd-unlabel} in~\cite{read1}. Figure~\ref{f-pd_molecular} shows the unlabeled nonempty point-determining graphs on $n\le 5$ vertices.
\begin{figure}[ht]
  \begin{center}
  \includegraphics[width=15cm]{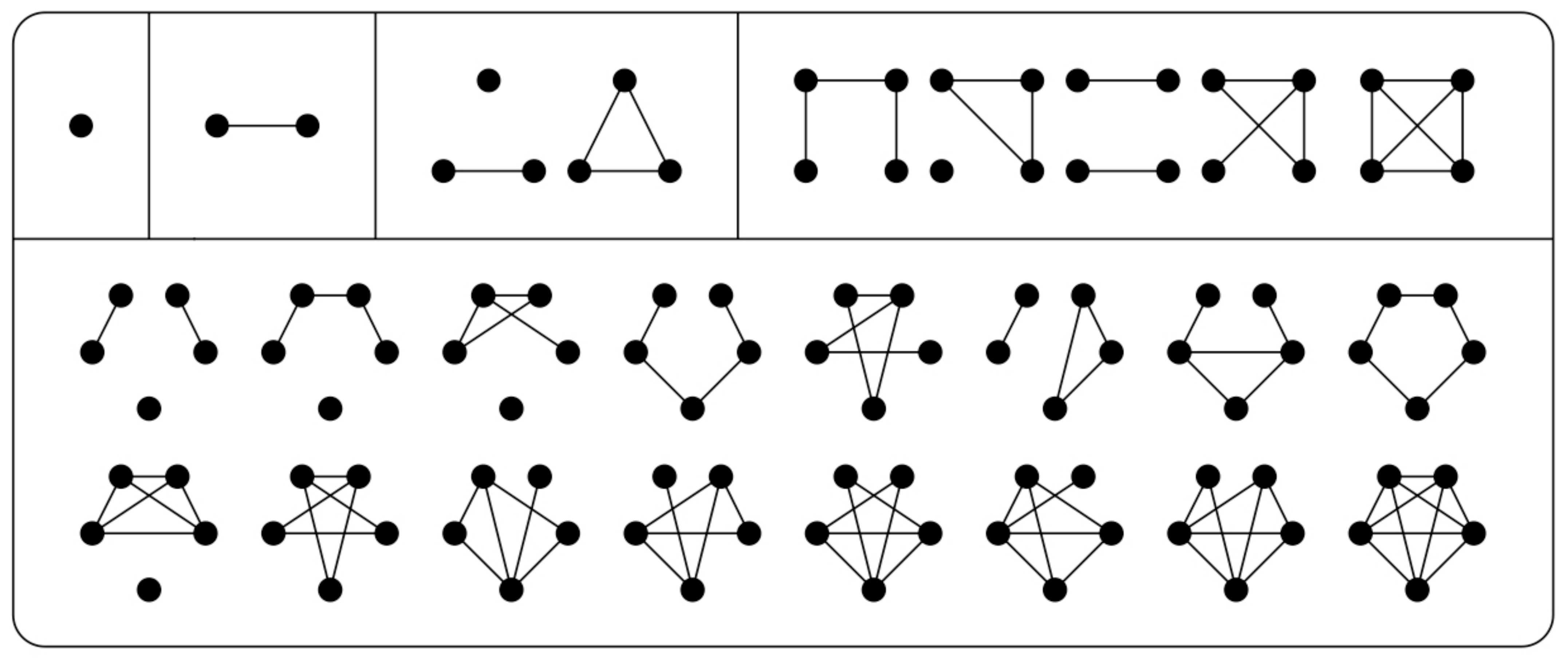}
  \end{center}
  \caption{\label{f-pd_molecular} Unlabeled nonempty point-determining graphs on at most 5 vertices.}
\end{figure}

Let $\mc{G}^c$ be the species of connected graphs. The observation that every graph is a set of connected graphs gives rise to the species identity
\begin{equation}\label{eqn-conngraph}
\mc{G}=\mc{E}\circ{\mc{G}^c},
\end{equation}
which can be written as $\mc{G}^c=(1+X)^c\circ\mc{G}_+.$

Connected point-determining graphs and connected co-point-determining graphs may be enumerated by looking at the connected components of point-determining graphs and co-point-determining graphs. In contrast to point-determining graphs and co-point-determining graphs, the species of connected point-determining graphs and connected co-point-determining graphs are not isomorphic.

\begin{thm}
\label{thm-connpdcpd}
For the species $\mc{P}$ of point-determining graphs, $\mc{Q}$ of co-point-determining graphs, $\mc{P}^c$ of connected point-determining graphs, and $\mc{Q}^c$ of connected co-point-determining graphs, we have
\begin{equation}\label{eqn-connpdcpd}
\mc{P}=\mc{Q}=(1+X)\cdot (\mc{E}\circ\mc{P}^c_{\ge 2})=\mc{E}\circ\mc{Q}^c.
\end{equation}
\end{thm}

\begin{proof}
A point-determining graph can have at most one isolated vertex, and its other connected components are connected point-determining graphs with at least two vertices. Therefore, $\mc{P}=(1+X)\cdot (\mc{E}\circ\mc{P}^c_{\ge 2})$. On the other hand,  a graph is co-point-determining if and only if all its connected components are. Therefore, $\mc{Q}=\mc{E}\circ\mc{Q}^c.$
\end{proof}

\begin{lem}
\label{lem-especies}
Let species $\Phi$ and $\Psi$ satisfy $\mc{E}(\Phi)=\mc{E}(\Psi)$.
Then $\Phi=\Psi$.
\end{lem}

\begin{proof}
It follows from $\mc{E}_+(\Phi)=\mc{E}_+(\Psi)$ that $(1+X)^c\circ \mc{E}_+(\Phi)=(1+X)^c\circ
\mc{E}_+(\Psi)$. Since $(1+X)^c \circ \mc{E}_+ = X,$ we have $\Phi=\Psi$.
\end{proof}

\begin{cor}
\label{cor-conncpd}
The species $\mc{Q}^c$ of connected  co-point-determining graphs and $\mc{G}^c$ of connected graphs satisfy
\begin{equation}\label{eqn-conncpdconngamma}
\mc{Q}^c=\mc{G}^c\circ(1+X)^c.
\end{equation}
\end{cor}

\begin{proof}
Since the composition of species is associative~\cite[p.~53, Exercise 1]{species-book}, we deduce from~\eqref{eqn-conngraph} and~\eqref{eqn-qgamma} that $$\mc{Q}=\mc{G}\circ(1+X)^c=\mc{E}\circ\mc{G}^c\circ(1+X)^c.$$
The result follows immediately from Theorem~\ref{thm-connpdcpd} and Lemma~\ref{lem-especies}.
\end{proof}

Theorem~\ref{thm-connpdcpd} gives $\mc{Q}^c=(1+X)^c\circ\mc{P}_+.$ We have the following formulas for $\mc{Q}^c$:
\begin{align}
\mc{Q}^c(x)&=\log(\mc{P}(x)), \notag\\
\wt{\mc{Q}^c}(x)&=\sum_{k\ge 1} \frac{\mu(k)}{k}\,\log(\wt{\mc{P}}(x^k)), \notag\\
 Z_{\mc{Q}^c}&=\sum_{k \ge 1}\frac{\mu(k)}{k}\,\log(Z_{\mc{P}}\circ p_k). \notag
\end{align}

A consequence of Theorem~\ref{thm-connpdcpd} is $$(1+X)\cdot(\mc{E}\circ\mc{P}^c_{\ge 2})=\mc{E}\circ((1+X)^c+\mc{P}^c_{\ge 2})=\mc{E}\circ\mc{Q}^c.$$ Therefore, Lemma~\ref{lem-especies} gives
\begin{equation}\label{eqn-gammadifference}
(1+X)^c=\mc{Q}^c-\mc{P}^c_{\ge 2}.
\end{equation}
We have the following functional
equations relating the associated series of $\mc{P}^c$ and those
of $\mc{Q}^c$:
\begin{align}
\mc{Q}^c(x)  -\mc{P}^c(x) &=  \log(1+x) -x, \notag\\
\wt{\mc{Q}^c}(x)-\wt{\mc{P}^c}(x)&=  -x^2, \label{eqn-connpdcpd-ordinary}\\
Z_{\mc{Q}^c}- Z_{\mc{P}^c} &=    \sum_{k\ge 1}\frac{\mu(k)}{k} \log (1+p_k) -p_1. \notag
\end{align}

Note that the only unlabeled connected graph on two vertices is point-determining, and this accounts for the right-hand side of~\eqref{eqn-connpdcpd-ordinary}. Thus~\eqref{eqn-connpdcpd-ordinary} says that for $n>2$ there are as many   unlabeled connected point-determining as unlabeled connected co-point-determining graphs. A combinatorial bijection between these two sets might be interesting.

Proposition 7 of~\cite[p.~122]{species-book} states that every virtual species $\Phi$ can be written uniquely in its \emph{reduced form} \[\Phi=\Phi^{+}-\Phi^{-},\] where $\Phi^{+}$ and $\Phi^{-}$ are species with no molecular components in common. Now~\eqref{eqn-gammadifference} gives a way to write the virtual species $(1+X)^c$ as the difference of two species. However $\mc{Q}^c-\mc{P}^c_{\ge 2}$ is not the reduced form of $(1+X)^c$, since $\mc{Q}^c$ share the same molecular components as $\mc{P}^c_{\ge 2}$. For example, we can write the first few terms of the molecular decompositions of $\mc{P}^c$ and  $\mc{Q}^c$ as follows:
\begin{align*}
\mc{P}^c &=X+\mc{E}_2+\mc{E}_3+\big(\mc{E}_2\circ X^2+X^2\mc{E}_2+\mc{E}_4\big)+\cdots\\
\mc{Q}^c &=X+X\mc{E}_2+\big(\mc{E}_2\circ X^2+X\mc{E}_3+\mc{E}_2\circ\mc{E}_2\big)+\cdots
\end{align*}

For any finite set $U$, the intersection $\mc{P}^c[U]\cap\mc{Q}^c[U]$ is the set of connected \emph{bi-point-determining graphs} on $U$, denoted $\mc{B}^c[U]$ (enumeration of bi-point-determining graphs is carried out in section~\ref{sec-bipd}). The species $\mc{B}^c$ is a \emph{subspecies} (see~\cite[p.~120]{species-book}) of both $\mc{P}^c$ and $\mc{Q}^c$, and $$(1+X)^c= (\mc{Q}^c-\mc{B}^c)-(\mc{P}^c-\mc{B}^c).$$ However, further examination shows that this is still not a reduced form of $(1+X)^c$. 

\section{Graphs Without Endpoints}
\label{sec-noend}
Let $\mc{H}(X,Y)$ be the $2$-sort species of connected graphs in which every vertex of degree one has sort $Y$ and every other vertex has sort $X$. 

\begin{thm}
\label{thm-2sortnoend}
The $2$-sort species $\mc{H}(X,Y)$ satisfies
\begin{equation}\label{eqn-ira2sort}
\mc{H} (X,X+Y)=\mc{G}^c\circ(X\mc{E}(Y)) + \mc{E}_2(Y),
\end{equation}
where $\mc{G}^c$ is the species of connected graphs, and $\mc{E}_2$ is the species of $2$-element sets.
\end{thm}

\begin{proof}
An $\mc{H}(X,X+Y)$-structure is a connected graph in which every vertex of degree one has sort either $X$ or $Y$, and every other vertex has sort $X$. Such a graph either is a graph with two vertices, both of sort $Y$, which is an $\mc{E}_2(Y)$-structure, or has at least one vertex of sort $X$.  An $\mc{H}(X, X+Y)$-structure with at least one vertex of sort $X$ consists of a connected graph whose vertices all have sort $X$, together with some additional vertices of sort $Y$, each adjacent to one of the vertices of sort $X$. 

An $X\mc{E}(Y)$-structure is a singleton $X$-structure connected to a set, possibly empty, of $Y$-structures. (See Figure~\ref{f-noend_2sort}.)
\begin{figure}[ht]
\begin{center}
\includegraphics[width=11cm]{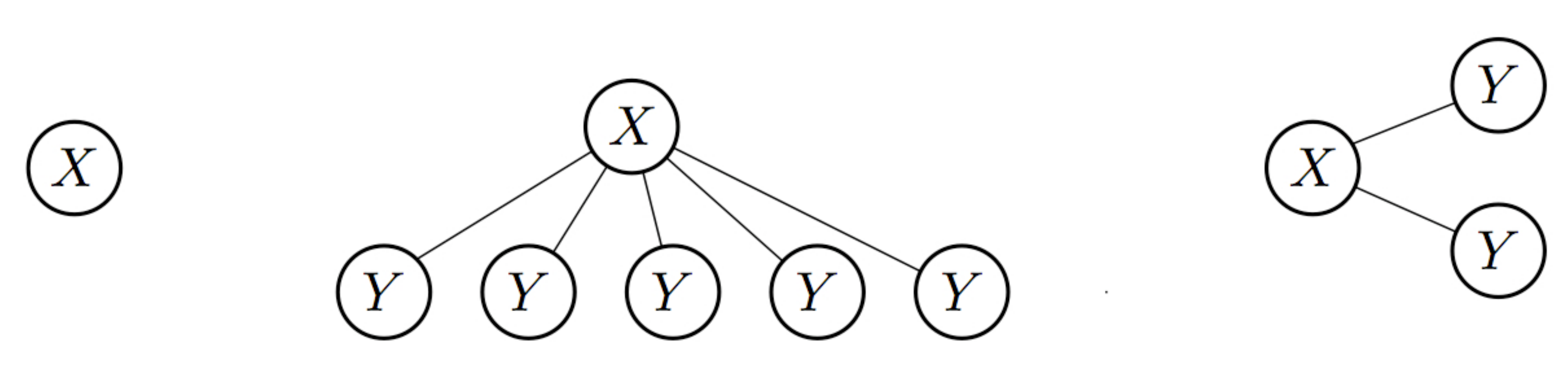}
\end{center}
\caption{\label{f-noend_2sort} Some $X\mc{E}(Y)$-structures.}
\end{figure}
We get a $\mc{G}^c\circ(X\mc{E}(Y))$-structure by replacing each vertex of a connected graph with an $X\mc{E}(Y)$-structure. Such a graph is the same as an $\mc{H}(X, X+Y)$-structure with at least one vertex of sort $X$. 
See Figure~\ref{f-noend_2sort_pf} for an illustration of an  $\mc{H}(X,X+Y)$-structure decomposed into a connected graph with each vertex replaced with an $X\mc{E}(Y)$-structure.
\end{proof}

\begin{figure}[ht]
\begin{center}
\includegraphics[width=17cm]{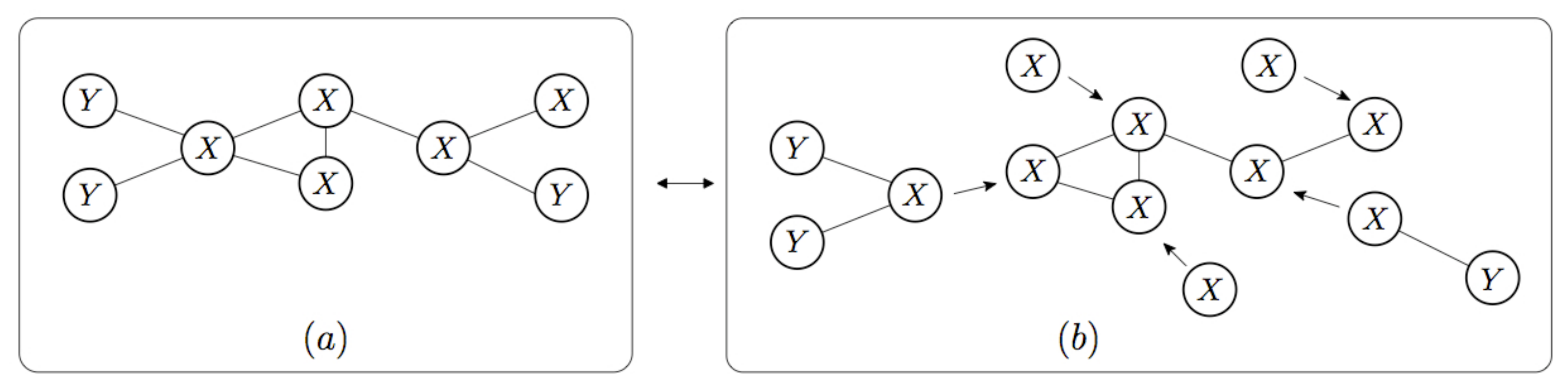}
\end{center}
\caption{\label{f-noend_2sort_pf}$\mc{H} (X,X+Y)-\mc{E}_2(Y)=\mc{G}^c\circ(X\mc{E}(Y))$.}
\end{figure}

Let $\mc{M}$  be the species of \emph{graphs without endpoints}, i.e., graphs without vertices of degree one, including the empty graph. Let $\mc{M}^c$ be the species of connected graphs without endpoints. Since a graph has no endpoints if and only if all of its connected components have no endpoints,  we have
\begin{equation}\label{eqn-noend-conn}
\mc{M}=\mc{E}\circ\mc{M}^c.
\end{equation}

Note that $\mc{M}^c(X)=\mc{H}(X,0)$. Replacing $Y$ with $-X$ in~\eqref{eqn-ira2sort}, we get an expression for the species $\mc{M}^c$ in terms of virtual species
\begin{equation}\label{eqn-ira}
\mc{M}^c=\mc{H}(X,0)=\mc{G}^c\circ(X\mc{E}(-X))+\mc{E}_2(-X).
\end{equation}
The type generating series for $\mc{E}_2(-X)$ is $0$, and the cycle index series for $\mc{E}_2(-X)$ is $(p_1^2-p_2)/2$. We get the associated generating series of $\mc{M}^c$:
\begin{align}
\mc{M}^c(x) &= \mc{G}^c(xe^{-x})+\frac{x^2}{2}, \label{eqn-connnoend-label}\\
\wt{\mc{M}^c}(x) &=Z_{\mc{G}^c}(\wt{X\mc{E}(-X)}(x), \wt{X\mc{E}(-X)}(x^2),\dots), \label{eqn-connnoend-unlabel}\\
Z_{\mc{M}^c} &=Z_{\mc{G}^c}(Z_{X\mc{E}(-X)}\circ p_1, Z_{X\mc{E}(-X)}\circ p_2,\dots)+\frac{1}{2}(p_1^2-p_2).\notag
\end{align}

Formula~\eqref{eqn-connnoend-label} was given by Wright~\cite[p.~206, Theorem 1]{wright} and by Goulden and Jackson~\cite[p.~180, Theorem 1]{goulden} and leads to the exponential generating series of the species $\mc{M}$ of graphs without endpoints
$$\mc{M}(x)=  \biggl(\sum_{n\ge 0}2^{\big(\aatop{n}{2}\big)}\frac{(xe^{-x})^n}{n!} \biggr) \exp\biggl(\frac{x^2}{2}\biggr).$$

\begin{cor}
\label{cor-noendpd}
For the species $\mc{M}$ of graphs without endpoints, $\mc{M}^c$ of connected graphs without endpoints, $\mc{Q}$ of co-point-determining graphs, and $\mc{Q}^c$ of connected co-point-determining graphs, we have the following identities for their type generating series:
\begin{align}
\wt{\mc{M}}(x) &=\wt{\mc{Q}}(x).\label{eqn-mq-unlabel}\\
\wt{\mc{M}^c}(x) &=\wt{\mc{Q}^c}(x), \label{eqn-mcqc-unlabel}
\end{align}
\end{cor}

\begin{proof}
Let $\mc{B}=X\mc{E}(-X)$. The type generating function of $\mc{B}$ is
\begin{align*}
\wt{\mc{B}}(x) &= x Z_{\mc{E}}(-x, -x^2, \dots) =x \exp\biggl(- \sum_{n\ge 1} \frac{x^n}{n}\biggr) =x-x^2=\wt{(1+X)^c}(x),
\end{align*}
where the virtual species $(1+X)^c$ is the compositional inverse of $\mc{E}_+$. Recall~\eqref{eqn-conncpdconngamma}:
$$\wt{\mc{Q}^c}(x)=Z_{\mc{G}^c}(\wt{(1+X)^c}(x), \wt{(1+X)^c}(x^2),\dots).$$
We get~\eqref{eqn-mcqc-unlabel} from~\eqref{eqn-connnoend-unlabel}. The equation~\eqref{eqn-mq-unlabel} follows from~\eqref{eqn-connpdcpd} and~\eqref{eqn-noend-conn}.
\end{proof}

Kilibarda~\cite{kilibarda} gave a bijective proof of Corollary~\ref{cor-noendpd}. 

Robinson gave a formula~\cite[p.~353, Theorem 8]{Robinson-nonseparable} (see also~\cite[p.~191, equation (8.7.11)]{harary}) for enumerating connected graphs without endpoints. Let $\mc{A}$ be the species of trees, and let $\mc{A}^r$ be the species of rooted trees. Robinson's formula may be expressed in terms of species as
\begin{equation}\label{eqn-speciesrobinson}
\mc{G}^c=\mc{A}+(\mc{M}^c-X)\circ\mc{A}^r,
\end{equation}
which is equivalent to our next result.

\begin{cor}
\label{cor-robinson}
We have another expression for the species $\mc{M}^c$ in terms of virtual species
\begin{equation}
\mc{M}^c=X+(\mc{G}^c-\mc{A})\circ\mc{B},\notag
\end{equation}
where $\mc{B}=X\mc{E}(-X)$, and $\mc{G}^c$ and $\mc{A}$ denote the species of connected graphs and trees, respectively.
\end{cor}

\begin{proof}
We start with the dissymmetry theorem for trees~\cite[p.~280, Theorem 1]{species-book} $$\mc{A}^r+\mc{E}_2\circ\mc{A}^r=\mc{A}+(\mc{A}^r)^2,$$ and rewrite it in terms of virtual species $$\mc{A}=(X+\mc{E}_2-X^2)\circ \mc{A}^r.$$ We apply the identity for virtual species~\cite[p.~128]{species-book} $$\mc{E}_2(-X)=X^2-\mc{E}_2(X),$$ and get $$\mc{A} = (X-\mc{E}_2(-X))\circ \mc{A}^r.$$ Since $\mc{B}$ is the compositional inverse of the species $\mc{A}^r$ of rooted trees (see~\cite[p.~132]{species-book}), we have $$\mc{A}\circ\mc{B}=X-\mc{E}_2(-X).$$ The result follows from \eqref{eqn-ira}.
\end{proof}

Equation~\eqref{eqn-speciesrobinson} also appeared in~\cite[p.~303, Example 5]{species-book} as an application of the dissymmetry theorem for graphs~\cite[p.~301, Theorem 3]{species-book}.

\section{Bi-Point-Determining Graphs}
\label{sec-bipd}
\begin{dfn}
\label{dfn-cograph}
A \emph{cograph}, also called a \emph{complement-reducible graph} is defined recursively as follows (see~\cite{corneil}):
\begin{enumerate}[(i)]
\item A graph on a single vertex is a cograph.
\item For a set of cographs $\{G_1, \dots, G_n\}$, their union $E_n|_{G_1, \dots, G_n}$ is also a cograph.
\item If $G$ is a cograph, then so is its complement.
\end{enumerate}
\end{dfn}

Note that the complement of $E_n|_{G_1, \dots, G_n}$ is $K_n|_{H_1, \dots, H_n}$ where each $H_i$ is the complement of $G_i$. It follows from the definition that the join of a set of cographs is a cograph. Let $\mc{C}$ be the species of cographs, and let $\mc{C}^c$ be the species of connected cographs.

\begin{lem}
\label{lem-cograph}
The species $\mc{C}$ of cographs satisfies the combinatorial equality $$\mc{C}=\mc{E}_+\circ\biggl(\frac{\mc{C}+X}{2}\biggr).$$
\end{lem}

\begin{proof}
Since a cograph is connected if and only if its complement is a disconnected cograph, the species $\mc{C}^c$ of connected cographs is isomorphic to the species $\mc{C}-\mc{C}^c+X$ of disconnected cographs. That is, $\mc{C}^c=\mc{C}-\mc{C}^c+X,$ which gives $$\mc{C}^c=\frac{\mc{C}+X}{2}.$$ On the other hand, each cograph consists of at least one connected component each of which is a $\mc{C}^c$-structure. This gives $$\mc{C}=\mc{E}_+\circ\mc{C}^c=\mc{E}_+\circ\biggl(\frac{\mc{C}+X}{2}\bigg).$$
\end{proof}

Note that the species $\mc{C}^c$ of connected cographs satisfies $\mc{C}^c = X + \mc{E}_{\ge 2}\circ \mc{C}^c$, so 
$\mc{C}^c$ is isomorphic to the species of \emph{phylogenetic trees} \cite{corneil, phylogenetic1, schroder}, which are rooted trees with labeled leaves and unlabeled internal vertices, in which every internal vertex has at least two children. Labeled and unlabeled cographs have been counted by Guruswami \cite{guruswami}.

\begin{lem}
\label{lem-cographinverse}
The compositional inverse of the species $\mc{C}$ of cographs is $$\mc{C}^{\inverse}=2(1+X)^c-X.$$
\end{lem}

\begin{proof}
Recall that $(1+X)^c\circ \mc{E}_+=X.$ It follows from Lemma~\ref{lem-cograph} that
$$(1+X)^c\circ\mc{C}=\frac{\mc{C}+X}{2}.$$
Therefore, $$2(1+X)^c\circ\mc{C}-\mc{C}=(2(1+X)^c-X)\circ\mc{C}=X.$$
Since the species $\mc{C}$ satisfies $\mc{C}_0=0, \mc{C}_1=X$, by Proposition 19 on~\cite[p.~130]{species-book} there exists a unique virtual species $\mc{C}^{\inverse}$ such that
$$\mc{C}^{\inverse}\circ\mc{C}=\mc{C}\circ\mc{C}^{\inverse}=X.$$
The result follows.
\end{proof}

A \emph{bi-point-determining graph} is a point-determining graph whose complement is also point-determining. As we noted earlier, a graph is bi-point-determining if and only if its automorphism group contains no transpositions, and it is not difficulty to show that the automorphism group of a bi-point-determining graph cannot contain any 3-cycles or 4-cycles.

The following theorem is the key to enumerating bi-point-determining graphs.

\begin{thm}
\label{thm-bipd}
The species $\mc{G}$ of graphs is the composition of the species $\mc{B}$ of bi-point-determining graphs and $\mc{C}$ of cographs. That is, $$\mc{G}= \mc{B}\circ\mc{C}.$$
\end{thm}

First we prove a lemma.
Following \cite{corneil}, let us call two distinct vertices \emph{weak siblings} if they have the same neighborhood and \emph{strong siblings} if they have the same closed neighborhood.
\begin{lem}
\label{lem-siblings}
In any graph, if $t$ and $u$ are weak siblings and $v$ and $w$ are strong siblings, then the sets $\{t,u\}$ and $\{v,w\}$ are disjoint.
\end{lem}
\begin{proof}
Suppose that $t$ and $u$ are weak siblings and that $t$ and $v$ are strong siblings. Since $t$ and $u$ have the same neighborhood, $t$ and $u$ are not adjacent. Since $t$ and $v$ have the same closed neighborhood, $u$ and $v$ are not adjacent but $t$ and $v$ are adjacent.  Now, since $t$ and $u$ have the same neighborhood, $u$ and $v$ are adjacent, a contradiction.
\end{proof}

\begin{proof}
[Proof of Theorem~\ref{thm-bipd}]
We show that every graph can be expressed uniquely as a superimposition of a bi-point-determining graph on a set of cographs. Let $K$ be any graph. We consider the set $S(K)$ of pairs $(G, \{H_1, \cdots, H_m\})$ such that $G|_{H_1,\dots, H_m}=K$ and $H_1,\dots, H_k$ are cographs. First, we note that $S(K)$ is nonempty, because it contains the pair $(K_0, V_0)$, where $V_0$ is the set of singleton graphs on the vertices of $K$ and $K_0$ is the graph obtained from $K$ by replacing each vertex $v$ of $K$ with $\{v\}$. 

Next we make $S(K)$ into a digraph. We say that there is an edge from $(G,\{H_1, \dots, H_m\})$ to $(G', \{H_1,\dots , H_{m-2},H_{m-1}'\})$ if $V(H_{m-1})$ and $V(H_{m})$, as vertices of $G$, have either the same neighborhood or the same closed neighborhood, $G'$ is obtained from $G$ by replacing vertices $V(H_{m-1})$ and  $V(H_{m})$ with the new vertex $V(H_{m-1})\cup V(H_{m})$ (with the same neighbors that $V(H_{m-1})$ and  $V(H_{m})$ had in $G$), and $H_{m-1}'$ is the induced subgraph of $K$ on $V(H_{m-1})\cup V(H_{m})$.  Note that $H_{m-1}'$ is a cograph since it is either a union or join of two cographs. 

It is clear that the sinks of this digraph are the pairs $(G, \{H_1, \cdots, H_m\})$ in which $G$ is bi-point-determining. Thus to prove the theorem we need to show that $S(K)$ has a unique sink. This will follow from the ``diamond lemma" of Newman 
\cite{newman} if we can prove the following two properties of $S(K)$.
\begin{enumerate}
\item $S(K)$ has a unique source, the pair $(K_0,V_0)$ defined above.
\item If $\alpha$, $\beta$, and $\gamma$ are vertices of $S(K)$ such that there is an edge in $S(K)$ from $\alpha$ to $\beta$ and an edge from $\alpha$ to $\gamma $, then there is a vertex $\delta$ of $S(K)$ such that there is an edge from $\beta$ to $\delta$ and an edge from $\gamma$ to $\delta $.
\end{enumerate}

To prove (1), we note that $(K_0,V_0)$ is the only pair  $(G, \{H_1, \cdots, H_m\})$ in $S(K)$ for which $H_1$, \dots, $H_m$ are all singleton graphs. So it is enough to show that if $(G, \{H_1, \cdots, H_m\})\in S(K)$ and $H_m$ is not a singleton graph, then $(G, \{H_1, \cdots, H_m\})$ has a predecessor in $S(K)$ (i.e., there is an edge from some element of $S(K)$ to $(G, \{H_1, \cdots, H_m\})$).  Since $H_m$ is a cograph that is not a singleton graph, $H_m$ can be expressed as either a join or union of two nonempty cographs. Suppose first that $H_m$ is the union of $H_m'$ and $H_{m+1}'$. Then $(G', \{H_1, \cdots, H_{m-1}, H_m', H_{m+1}'\})$ is a predecessor of $(G, \{H_1, \cdots, H_m\})$, where $G'$ is obtained from $G$ by ``splitting" vertex $V(H_m)$ into vertices $V(H_m')$ and $V(H_{m+1}')$; these new vertices are adjacent in $G'$ to all the neighbors of $V(H_m)$ in $G$ but not to each other.  The case in which $H_m$ is a join is similar.

To prove property (2), we note that if $\beta$ and $\gamma$ are both obtained from $\alpha$ by amalgamating pairs of vertices with the same neighborhood, or are both obtained by amalgamating pairs with the same closed neighborhood then the existence of $\delta$ is clear. If $\beta$ is obtained by amalgamating a pair of vertices with the same neighborhood  and $\gamma$ is obtained by amalgamating a pair of vertices with the same closed neighborhood, then by Lemma \ref{lem-siblings}, the four amalgamated vertices are all distinct, and thus the existence of $\delta$ is again clear.
\end{proof}

\begin{cor}
\label{cor-enumbipd}
In terms of virtual species, the species of bi-point-determining graphs $\mc{B}$ is related to the species of graphs $\mc{G}$ and the virtual species $(1+X)^c$ in the following way:
\begin{equation}\label{eqn-evaluateR}
\mc{B}=\mc{G}\circ(2(1+X)^c -X).
\end{equation}
\end{cor}

\begin{proof}
It follows immediately from Theorem~\ref{thm-bipd} and Lemma~\ref{lem-cographinverse}.
\end{proof}

Equation~\eqref{eqn-evaluateR} gives rise to identities for
computing the associated series of $\mc{B}$:
\begin{align}
\mc{B}(x)&=\mc{G}(2\log(1+x)-x), \notag\\
\wt{\mc{B}}(x)&=Z_{\mc{G}}(x-2x^2, x^2-2x^4, \dots),\notag \\
Z_{\mc{B}}&=Z_{\mc{G}}\biggl( 2\sum_{k\ge 1}\frac{\mu_k}{
k}\,\log(1+p_k)-p_1, 2\sum_{k\ge 1}\frac{\mu_k }{
k}\,\log(1+p_{2k})-p_2, \dots \biggr). \notag
\end{align}

There are no bi-point-determining graphs on $3$ vertices. The unlabeled bi-point-determining graphs with $4$ or $5$ vertices
are shown in Figure~\ref{f-bipd_unlabel}.
\begin{figure}[ht]
\begin{center}
\includegraphics[width=15cm]{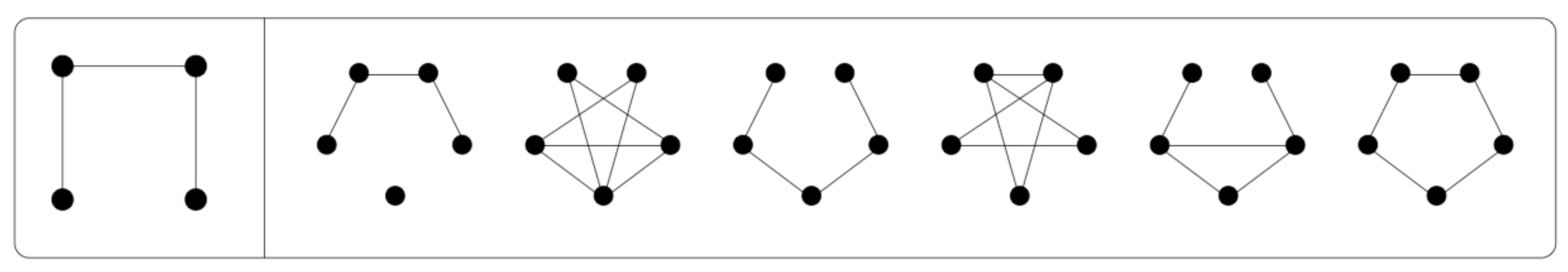}
\end{center}
\caption[\ \ Unlabeled bi-point-determining graphs.]{\label{f-bipd_unlabel} Unlabeled bi-point-determining graphs on $n$ vertices, $n=4,5$.}
\end{figure}

Next we enumerate connected bi-point-determining graphs.

\begin{cor}
\label{cor-connbipd}
The species $\mc{B}^c$ of connected bi-point-determining graphs is expressed, in terms of virtual species, by both of the following combinatorial identities
\begin{align}
\mc{B}^c &=(1+X)^c\circ\mc{B}_+-(1+X)^c+X,\label{eqn-bipdconn1}\\
\mc{B}^c &=\mc{G}^c\circ(2(1+X)^c-X)-(1+X)^c+X, \label{eqn-bipdconn2}
\end{align}
where $\mc{B}_+$ is the species of nonempty bi-point-determining graphs, $\mc{G}^c$ is the species of connected graphs, and $(1+X)^c$ is the compositional inverse of $\mc{E}_+$.
\end{cor}

\begin{proof}
A bi-point-determining graph can have at most one vertex of degree zero and the other connected components are connected bi-point-determining graphs with more than one vertex. Hence we have $$\mc{B}=(1+X)\cdot\mc{E}\circ(\mc{B}^c-X).$$ Since $(1+X)^c$ satisfies $1+X=\mc{E}\circ(1+X)^c$, we have $$\mc{B}_+=\mc{E}_+\circ((1+X)^c+\mc{B}^c-X),$$ and~\eqref{eqn-bipdconn1} follows.

Corollary~\ref{cor-enumbipd} gives $$\mc{B}_+=\mc{G}_+\circ\mc{C}.$$ Since $$(1+X)^c\circ\mc{G}_+=\mc{G}^c,$$ we have $$(1+X)^c\circ \mc{B}_+=\mc{G}^c\circ\mc{C}.$$ Now~\eqref{eqn-bipdconn2} follows from~\eqref{eqn-bipdconn1}.
\end{proof}

Corollary~\ref{cor-connbipd} allows us to enumerate the $\mc{B}^c$-structures based on our enumeration results on $\mc{B}$-structures or $\mc{G}^c$-structures. For example, the exponential generating series of $\mc{B}^c$ can be written in two ways:
\begin{align}
\mc{B}^c(x) &= \log(\mc{B}(x))-\log(1+x)+x, \notag \\
\mc{B}^c(x) &= \mc{G}^c(2\log(1+x)-x) - \log(1+x) +x. \notag
\end{align}

\section{Point-Determining Bicolored Graphs}\label{sec-2color}

A \emph{proper coloring} of a graph is an assignment of colors to the vertices of the graph where no two adjacent vertices are assigned the same color. A \emph{bicolorable graph} is a graph that can be properly colored with two colors. A \emph{bicolored graph} (or \emph{2-colored graph}) is a graph in which all vertices are properly bicolored. The enumeration of bicolorable and bicolored graphs was studied in~\cite{harary0},~\cite{harary5}, and~\cite{hanlon}.  For simplicity, we call the two colors in a bicolored graph white and black.

We denote by $\mc{G}(X,Y)$ the $2$-sort species  of bicolored graphs. To be more specific, for a two-set $U=(W,B)$, $\mc{G}[U]$ is a bicolored graph in which the vertices colored white are elements of $W$ and the vertices colored black are elements of $B$. Furthermore, for any multibijection $\sigma=(\tau_1, \tau_2): (W_1, B_1)\rightarrow(W_2, B_2),$ where $\tau_1$ is a bijection from $W_1$ to $W_2$ and $\tau_2$ is a bijection from $B_1$ to $B_2$, the transport of $\mc{G}(X,Y)$-structures along $\sigma$ is $\mc{G}[\sigma]$, which is a bijection from the set of bicolored graphs with vertex set $(W_1, B_1)$ to the set of bicolored graphs with vertex set $(W_2, B_2)$ that preserves the colors of all vertices. The isomorphism classes of (unlabeled) bicolored graphs are called \emph{color-non-isomorphic} bicolored graphs in~\cite{harary5}.

A quick observation is that in a bicolored graph, each edge must connect one vertex of sort $X$ and one vertex of sort $Y$, and hence there are $2^{mn}$ labeled bicolored graphs with $m$ white vertices and $n$ black vertices. Therefore, the exponential generating series of $\mc{G}(X,Y)$ is $$\mc{G}(x,y)=\sum_{m,n=0}^\infty 2^{mn} \frac{x^m}{m!}\,\frac{y^n}{ n!}.$$

\begin{thm}
\label{thm-2color}
Let $\mc{G}(X,Y)$ be the $2$-sort species of bicolored graphs. Then the cycle index of $\mc{G}(X,Y)$ is given by $$Z_{\mc{G}(X,Y)}=\sum_{m,n\ge 0}  \biggl( \sum_{\lambda \vdash m,\,\mu\vdash n} 2^{\,\sum_{i,j}\gcd(\lambda_i,\,\mu_j)}\,\frac{p_{\lambda}[x]}{z_{\lambda}}\, \frac{p_{\mu}[y]}{z_{\mu}}\biggr).$$
\end{thm}

\begin{proof}
Let $(\lambda, \mu)$ be an ordered pair of partitions, and let $(\sigma, \pi)$ be an ordered pair of permutations with $\sigma$ having cycle type $\lambda$ and $\pi$ having cycle type $\mu$. Let $\fix\,(\sigma,\pi)=\fix\,\mc{G}[\lambda,\mu]$ be the number of bicolored graphs fixed by $(\sigma,\pi)$.

To start with, we consider the simpler case when $\sigma$ is a $k$-cycle and $\pi$ is an $l$-cycle. Let $K_{k,l}$ denote the complete bipartite graph on $[k,l]$, and let $E(K_{k,l})$ be its edge set. Then $|E(K_{k,l})|=kl$. Without loss of generality, we let the labeling of left-hand side vertices of $K_{k,l}$ be $\{1,2,\dots, k\}$, and the labeling of right-hand side vertices of $K_{k,l}$ be $\{ 1', 2', \dots, l'\}$. Then each edge of $K_{k,l}$ is represented by an ordered pair $(i,j')$, for some $i\in[k]$ and $j \in[l]$. The pair of permutations $(\sigma, \pi)$ acts on the set $E(K_{k,l})$ by letting $\sigma$ act on the set $\{1,2,\dots, k\}$ and $\pi$ act on the set $\{ 1', 2', \dots, l'\}$. This action partitions the $kl$ edges of $K_{k,l}$ into orbits $\{A_1, A_2, \dots\}$. We observe that there are $\lcm(k,l)$ edges in each of the orbits, since all edges of the form $(i_r, j_r')$, where $i_r=\sigma^r(i)$ and $j_r=\pi^r(j)$ for some $r=1,2\dots,\lcm(k,l)-1$, are in the same orbit as the edge $(i,j')$, and hence this action of $(\sigma,\pi)$ on the set $E(K_{k,l})$ results in $(kl)/{\lcm(k,l)}=\gcd(k,l)$ orbits. Note that each bicolored graph with vertex set $[k,l]$ can be identified with a subset of $E(K_{k,l})$. If a subset $S$ of $E(K_{k,l})$ is fixed by the pair of permutations $(\sigma,\pi)$, then whenever an edge $(i,j')$ is in $S$, all edges in the same orbit as $(i,j')$ under the action of $(\sigma,\pi)$ on $E(K_{k,l})$ is in $S$ as well. This means that the number of bicolored graphs fixed by the pair of permutations $(\sigma,\pi)$ is the same as the number of subsets of $\{A_1, A_2, \dots, A_{\gcd(k,l)}\}$. Therefore, $$\fix\,(\sigma,\pi)=2^{\gcd(k,l)}.$$

For the general case, we write $\lambda=(\lambda_1,\lambda_2,\dots)$ and $\mu=(\mu_1,\mu_2,\dots)$. It is straightforward to see that each ordered pair $(\lambda_i, \mu_j)$, for some integers $i$ and $j$, gives rise to a factor $2^{\gcd{\lambda_i, \mu_j}}$ in the number $\fix\,(\sigma,\pi)$, and hence $$\fix\,\mc{G}[\lambda,\mu]=\prod_{i,j} 2^{\,\gcd(\lambda_i,\,\mu_j)}=2^{\sum_{i,j}\gcd(\lambda_i,\,\mu_j)}.$$
\end{proof}

Theorem~\ref{thm-2color} enables us to compute the associated series of $\mc{G}(X,Y)$. Equations~\eqref{eqn-2color-label},~\eqref{eqn-2color-unlabel}, and~\eqref{eqn-2colorconn-unlabel} appeared in~\cite{harary0} and~\cite{harary5}. The argument given in the proof of Theorem~\ref{thm-2color} also gives a way to count bicolored graphs by the number of edges. Let $b_{m,n}(x)$ be the ordinary generating function for bicolored graphs, in which $m$ vertices are colored white and $n$ vertices colored black, by the number of edges. We get the following expression for $b_{m,n}(x)$, which agrees with the result of Harary and Palmer~\cite[p.~95]{harary}: $$b_{m,n}(x)=\sum_{\lambda \vdash m,\,\mu\vdash n} \frac{1}{z_{\lambda}z_{\mu}} \, \prod_{k,l=1}^{m,n} \biggl(1+x^{\lcm(k,\,l)}\biggr)^{c_k(\lambda) c_l(\mu) \gcd(k,\,l)},$$ where $c_i(\lambda)$ denotes the number of parts in $\lambda$ with length $i$. As illustrated in Figure~\ref{f-b23_x4}, there are three unlabeled bicolored graphs with four edges and five vertices, two colored white, three colored black, hence the coefficient of $x^4$ in $b_{2,3}(x)$ is $3$.
\begin{figure}[ht]
\begin{center}
\includegraphics[width=8cm]{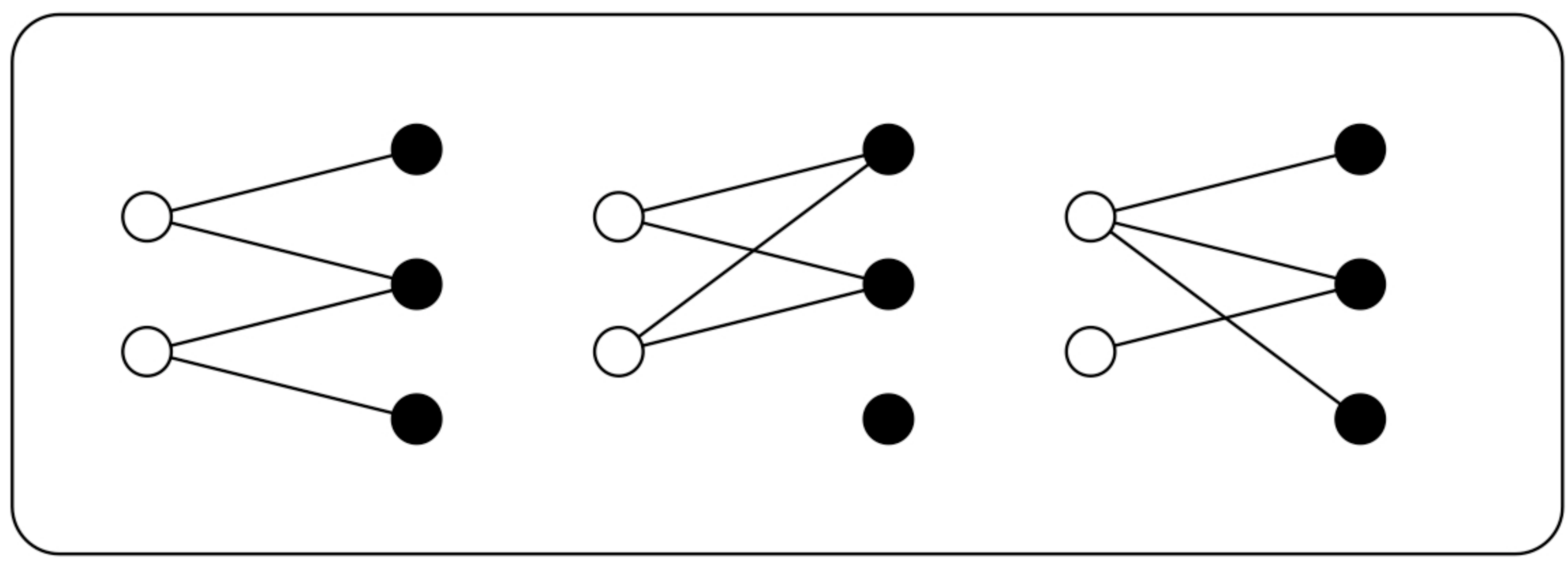}
\end{center}
\caption{\label{f-b23_x4} Unlabeled bicolored graphs with four edges and five vertices.}
\end{figure}

The canonical decomposition of a graph into connected components applies to bicolored graphs.

\begin{prop}
\label{prop-2colorconn}
The species $\mc{G}^c(X,Y)$ of connected bicolored graphs and the species $\mc{G}(X,Y)$ of bicolored graphs satisfy $$\mc{G}(X,Y)=\mc{E}\circ\mc{G}^c(X,Y).$$
\end{prop}

If follows that
\begin{equation}\label{eqn-2colorconn}
\mc{G}^c(X,Y)=(1+X)^c\circ\mc{G}_+(X,Y).
\end{equation}
Figure~\ref{f-2color4_conn} shows the unlabeled connected bicolored graphs with at most four vertices.
\begin{figure}[ht]
\begin{center}
\includegraphics[width=9cm]{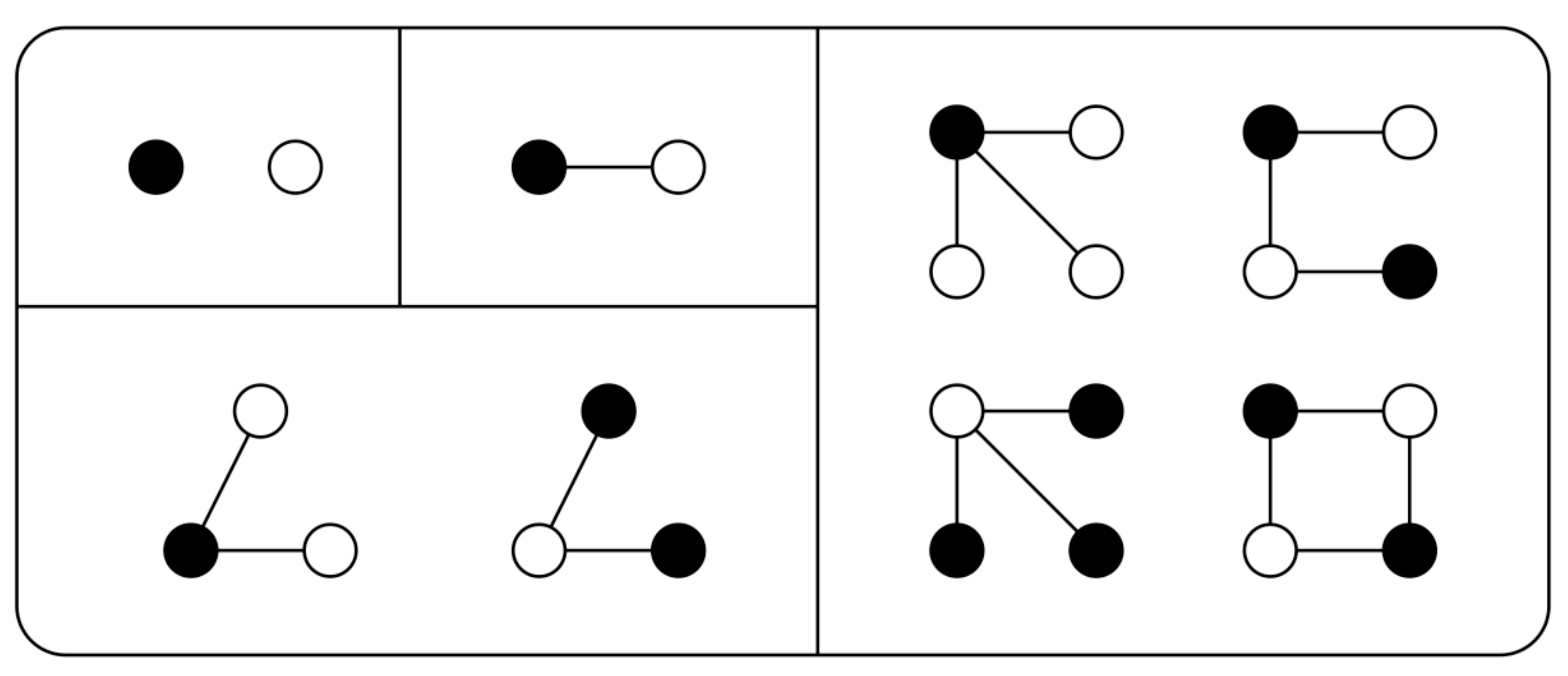}
\end{center}
\caption{\label{f-2color4_conn} Unlabeled connected bicolored graphs with $n$ vertices, $n\le 4$.}
\end{figure}

A bicolored graph is called \emph{point-determining} if the underlying graph is point-determining. A bicolored graph is called \emph{semi-point-determining} if all vertices of the same  color have distinct neighborhoods.  Note that the notion of co-point-determining bicolored graphs is not interesting, since  any two adjacent vertices  in a bicolored graph are colored differently, so that there is no vertex that could  be adjacent to both of them.

\begin{thm}
\label{thm-2colorpd}
For the species $\mc{P}(X,Y)$ of bicolored point-determining graphs, $\mc{P}^s(X,Y)$ of bicolored semi-point-determining graphs, and $\mc{P}^c(X,Y)$  of bicolored connected point-determining graphs, we have
\begin{align}
\mc{P}^s(X,Y) &=(1+X)(1+Y)\,\mc{E}\circ\mc{P}^c_{\ge 2}(X,Y), \label{formula_pd2c1}\\
\mc{P}(X,Y) &=(1+X+Y)\,\mc{E}\circ\mc{P}^c_{\ge 2}(X,Y). \label{formula_pd2c2}
\end{align}
\end{thm}

\begin{proof}
Let $K$ be a bicolored semi-point-determining graph.
We observe that a connected component of $K$ could be either a single vertex colored white, a single vertex colored black, or a bicolored connected point-determining graph with at least two vertices. At the same time, $K$ can have at most one isolated vertex colored with each color, due to the fact that all vertices in $K$ of the same color must have distinct neighborhoods. Equation~\eqref{formula_pd2c1} follows by translating the above into combinatorial equalities.
\begin{figure}[ht]
\begin{center}
\includegraphics[width=16cm]{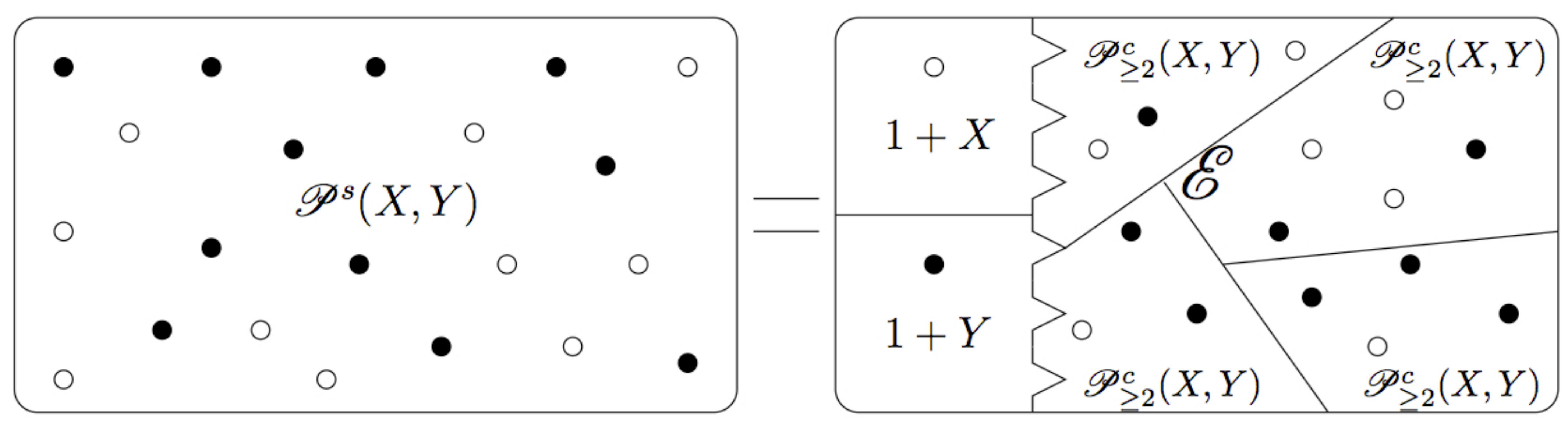}
\end{center}
\caption[\ \ The formula $\mc{P}^s(X,Y)=(1+X+Y)\,\mc{E}\circ\mc{P}^c_{\ge 2}(X,Y).$]{\label{f-bicolor_ppdconn}  $\mc{P}^s(X,Y)=(1+X+Y)\,\mc{E}\circ\mc{P}^c_{\ge 2}(X,Y).$}
\end{figure}

Let $H$ be a bicolored point-determining graph. As in the above discussion we see that a connected component of $H$ could be either a single vertex colored white, a single vertex colored black, or a bicolored connected point-determining graph with at least two vertices. But this time, since the underlying graph of $H$ is a point-determining graph, $H$ can have at most one isolated vertex. Hence the term $(1+X)(1+Y)$ in~\eqref{formula_pd2c1} is replaced with the term $1+X+Y$ in~\eqref{formula_pd2c2}.
\begin{figure}[ht]
\begin{center}
\includegraphics[width=16cm]{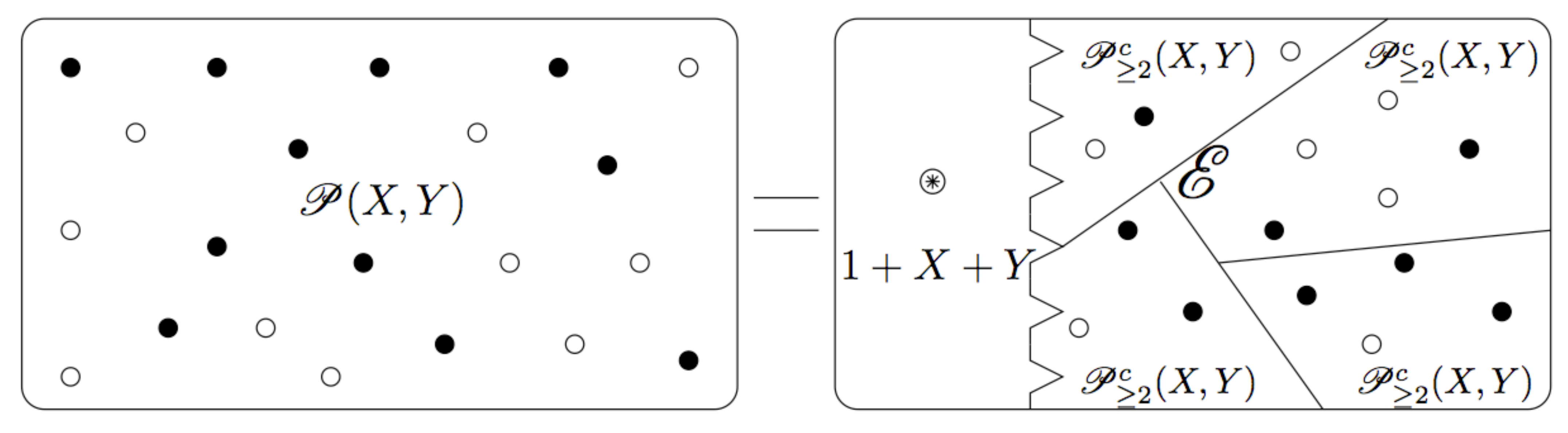}
\end{center}
\caption[\ \ The formula $\mc{P}(X,Y)=(1+X+Y)\,\mc{E}(\mc{P}^c_{\ge 2}(X,Y)).$]{\label{f-bicolor_pdconn}  $\mc{P}(X,Y)=(1+X+Y)\,\mc{E}(\mc{P}^c_{\ge 2}(X,Y)).$}
\end{figure}
\end{proof}

\begin{thm}
\label{thm-2colorps}
For the species $\mc{G}(X,Y)$ of bicolored graphs and $\mc{P}^s(X,Y)$ of bicolored semi-point-determining graphs, we have $$\mc{G}(X,Y)=\mc{P}^s(\mc{E}_+(X),\mc{E}_+(Y)).$$
\end{thm}

\begin{proof}
The proof uses the same idea as the proof of Theorem~\ref{thm-pdcpd}. To be more precise, given any bicolored graph, we define equivalence relations on the vertex sets by setting two same-colored vertices to be equivalent if they have the same neighborhoods, and get a new bicolored graph whose vertex set is the set of equivalence classes and the adjacency in the original graph is accordingly preserved. We observe that the resulting new graph is a bicolored semi-point-determining graph, and the rest is straightforward.
\end{proof}

Recall the virtual species $(1+X)^c$, the compositional inverse of $\mc{E}_+$. Theorem~\ref{thm-2colorps} gives $$\mc{P}^s(X,Y)=\mc{G}((1+X)^c(X), (1+X)^c(Y)),$$ which, together with~\ref{formula_pd2c1} and~\ref{formula_pd2c2}, allows us to compute the associated series of the species $\mc{P}^s(X,Y)$, $\mc{P}(X,Y)$, and $\mc{P}^c(X,Y)$. Figure~\ref{f-2color5_pd} shows the unlabeled nonempty point-determining bicolored graphs with at most five vertices and at least one vertex of each color.
\begin{figure}[ht]
\begin{center}
\includegraphics[width=15cm]{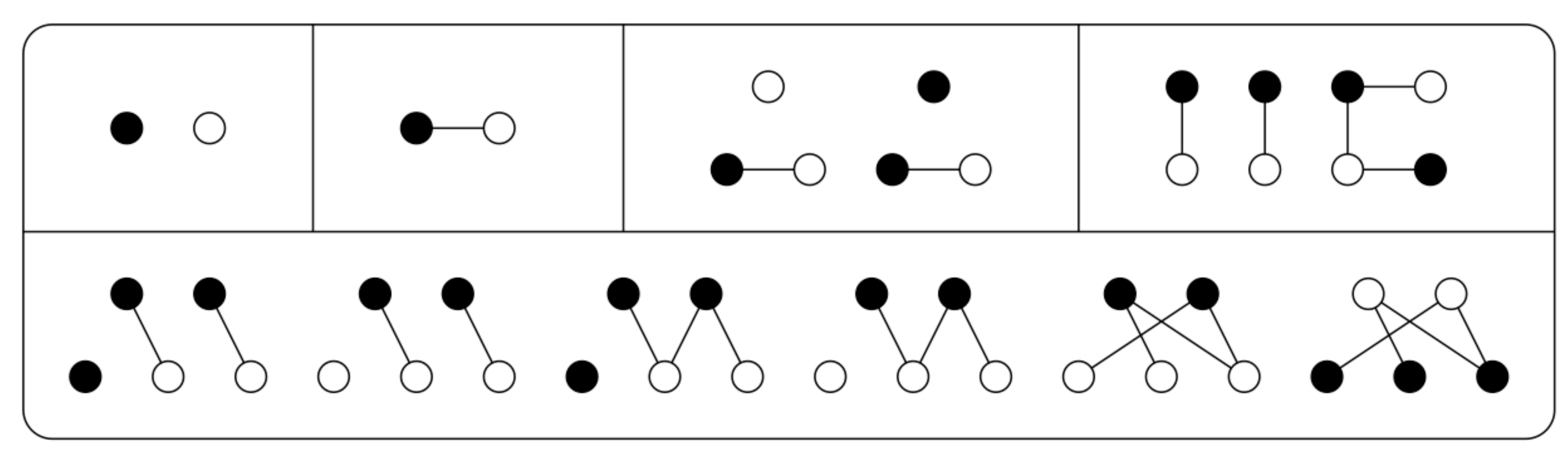}
\end{center}
\caption{\label{f-2color5_pd} Unlabeled nonempty point-determining bicolored graphs on $n$ vertices, $n\le 5$.}
\end{figure}

\appendix

\section{Index of Species}\label{appen-index}

\begin{tabular}[t]{ll}
$0$ &  empty species.\\[4pt]
$1$ &  characteristic of the empty set.\\[4pt]
$X$ &  species of singletons.\\[4pt]
$\mc{A}$ &  species of trees.\\[4pt]
$\mc{A}^r$ &  species of rooted trees.\\[4pt]
$\mc{B} / \mc{B}^c$ &  species of (connected) bi-point-determining graphs.\\[4pt]
$\mc{C} / \mc{C}^c$ &  species of (connected) cographs.\\[4pt]
$\mc{D}_n$ & molecular species of regular $n$-gons.\\[4pt]
$\mc{G} / \mc{G}^c$ &  species of (connected) simple graphs.\\[4pt]
$\mc{G}(X,Y) / \mc{G}^c(X,Y)$ &  $2$-sort species of (connected) bicolored graphs in which white vertices \\[2pt]
& are of sort $X$ and black vertices are of sort $Y$.\\[4pt]
$\mc{E}$ &  species of edgeless graphs.\\[4pt]
$\mc{K}$ &  species of complete graphs.\\[4pt]
$\mc{G}$ &  species of graphs.\\[4pt]
$\mc{G}(X,Y)$ &  $2$-sort species of bicolored graphs in which white vertices are of sort $X$\\[2pt]
                        & and black vertices are of sort $Y$.\\[4pt]
$(1+X)^c$ &  virtual species known as the compositional inverse of  $\mc{E}_+$. \\[4pt]
$\mc{H}(X,Y)$ & $2$-sort species of connected graphs in which vertices of degree one have sort $Y$  \\[2pt]
& and all other vertices have sort $X$.\\[4pt]
$\mc{M} / \mc{M}^c$ & species of (connected) graphs with no endpoints.\\[4pt]
$\mc{P} / \mc{P}^c$ &  species of (connected) point-determining graphs.\\[4pt]
$\mc{P}(X,Y)$ &  $2$-sort species of point-determining bicolored graphs.\\[4pt]
$\mc{P}^s(X,Y)$ &  $2$-sort species of semi-point-determining bicolored graphs.\\[4pt]
$\mc{P}^c(X,Y)$ &  $2$-sort species of connected
point-determining bicolored graphs.\\[4pt]
$\mc{Q} / \mc{Q}^c$ &  species of (connected) co-point-determining graphs.
\end{tabular}

\section{Cycle Indices and Molecular Decompositions}\label{appen-species}

In this section we give the first terms of the cycle indices and molecular decompositions of the species discussed in this paper. The cycle indices were computed with the help of  John Stembridge's SF package \cite{stembridge} for Maple. We have submitted the numbers of labeled and unlabeled structures to the Online Encyclopedia of Integer Sequences \cite{sloan} when they were not previously there, and thus rather than listing the numbers here, we give references to \cite{sloan}.

\subsection{The Species $\mc{P}$ of Point-Determining Graphs}

The numbers of labeled and unlabeled point-determining graphs are given in~\cite[A006024, A004110]{sloan}. Theorem~\ref{thm-pdcpd} allows us to compute the cycle index $Z_{\mc{P}}$ of point-determining graphs from the cycle index $Z_{\mc{G}}$ of graphs.
\begin{align*}
Z_{\mc{P}} &= 1+ p_1 + \biggl( \frac{1}{2}\, p_1^2+\,\frac{1}{2}\,p_2 \biggr) + \biggl(\frac{1}{3}\,p_3+p_1p_2+\,\frac{2}{3}\,p_1^3 + \,\frac{3}{2}\,p_1^2p_2 \biggr)   \notag\\
& \phantom{=} \ \,+ \biggl(\frac{1}{2}\,p_4+\,\frac{4}{3}\,p_1^4+p_2^2+\,\frac{2}{3}\,p_1p_3 \biggr) \notag\\
& \phantom{=} \ \, + \biggl( p_1^2 p_3+ \,\frac{49}{10}\, p_1^5+\,\frac{11}{3} \, p_1^3 p_2+\,\frac{1}{3}\,p_2p_3+\,\frac{3}{5}\,p_5+p_1p_4+\,\frac{9}{2}\,p_1p_2^2 \biggr) +\cdots.
\end{align*}

The molecular decomposition of $\mc{P}$ (see Figure~\ref{f-pd_molecular}) begins with
\begin{align*}
\mc{P} &= 1+X + \mc{E}_2 + ( X \cdot \mc{E}_2 + \mc{E}_3 )  + ( \mc{E}_2 \circ X^2 + X \cdot \mc{E}_3 + \mc{E}_2 \circ \mc{E}_2 +
  X^2 \cdot \mc{E}_2+ \mc{E}_4 )  \notag\\
  & \phantom{=}   + ( X\cdot \mc{E}_2\circ \mc{E}_2 + 5 X \cdot \mc{E}_2 \circ X^2  + 4 X^3\cdot \mc{E}_2 + X^2\cdot \mc{E}_3+ X \cdot \mc{E}_2 \cdot \mc{E}_2  + X\cdot \mc{E}_4 \notag\\
  & \phantom{=}  + \mc{D}_5 + \mc{E}_2\cdot \mc{E}_3   + \mc{E}_5 )+ \cdots,
\end{align*}
where $\mc{D}_5=X^5/D_5$ is the molecular species of  pentagons.

The numbers of labeled and unlabeled connected point-determining graphs are given in~\cite[A092430, A004108]{sloan}. Theorem~\ref{thm-connpdcpd} allows us to compute $Z_{\mc{Q}^c}$ and $Z_{\mc{P}^c}$ from $Z_{\mc{P}}$:
\begin{align*}
Z_{\mc{Q}^c} &=  p_1 + \biggl( \frac{1}{2}\, p_1p_2+ \,\frac{1}{2} \, p_1^3 \biggr)+\biggl( \frac{19}{24}\, p_1^4+\,\frac{3}{4}\, p_1^2p_2+\,\frac{1}{3}\,p_1p_3+\,\frac{7}{8}\,p_2^2+\,\frac{1}{4}\,p_4 \biggr)  \notag\\
& \phantom{=} \ \, +\biggl( \frac{7}{3}\,p_1^3p_2+\,\frac{1}{6}\,p_2p_3+\,\frac{77}{20}\,p_1^5+\,\frac{13}{4}\,p_1p_2^2+\,\frac{1}{2} \,p_1^2p_3+
\,\frac{2}{5}\,p_5+\,\frac{1}{2}\,p_1p_4 \biggr)+ \cdots.\\
Z_{\mc{P}^c} &= p_1+\biggl( \frac{1}{2}\,p_1^2+ \,\frac{1}{2}\,p_2  \biggr) +\biggl( \frac{1}{6}\,p_1^3+\,\frac{1}{3}\,p_3+\,\frac{1}{2}\,p_1p_2 \biggr)  \notag\\
& \phantom{=} \ \, +\biggl( \frac{25}{24}\,p_1^4+\,\frac{5}{8}\,p_2^2+\,\frac{3}{4}\,p_1^2p_2+\,\frac{1}{3}\,p_1p_3+\,\frac{1}{4}\,p_4 + \,\frac{1}{2}\,p_1^2p_3+\,\frac{13}{4}\,p_1p_2^2 \biggr) \notag\\
& \phantom{=}\ \, +\biggl( \frac{7}{3}\,p_1^3p_2+\,\frac{3}{5}\,p_5+\,\frac{73}{20}\,p_1^5+\,\frac{1}{6}\,p_2p_3+\,\frac{1}{2}\,p_1p_4 \biggr) + \cdots. 
\end{align*}

\subsection{The Species $\mc{M}$ of Graphs without Endpoints}

The numbers of labeled and unlabeled graphs without endpoints are given in~\cite[A059166, A004108]{sloan}. The equation $\mc{M}=\mc{E}\circ \mc{M}^c$ allows us to compute $Z_{\mc{M}}$:
\begin{align*}
Z_{\mc{M}} &=1+p_1+\biggl(\frac{1}{2}\,p_1^2+\,\frac{1}{2}\,p_2\biggr)+\biggl(\frac{1}{3}\,p_1^3+p_1p_2+\frac{2}{3}\,p_3\biggr)\\
&\phantom{=} +\biggl(\frac{13}{24}\,p_1^4 +\,\frac{7}{4}\,p_1^2p_2+\,\frac{4}{3}\,p_1p_3+\,\frac{5}{8}\,p_2^2+\,\frac{3}{4}\,p_4 \biggr)+\cdots
\end{align*}

The molecular decomposition of the species $\mc{M}$ begins with $$M= 1+X+\mc{E}_2 + 2\mc{E}_3 + (2X\mc{E}_3+2\mc{E}_4+\mc{D}_4)+\cdots,$$ where $\mc{D}_4=X^4/D_4$ is the molecular species of squares.

The numbers of labeled and unlabeled connected graphs without endpoints are given in~\cite[A059166, A004108]{sloan}. Equation~\eqref{eqn-ira} allows us to compute $Z_{\mc{M}^c}$:
\begin{align*}
Z_{\mc{M}^c} &= p_1 + \biggl( \frac{1}{6}\,p_1^3 +\,\frac{1}{2}\,p_1p_2 +\, \frac{1}{3}\, p_3\biggr)
+\biggl( \frac{5}{12}\,p_1^4 + p_1^2p_2+\,\frac{1}{3}\,p_1p_3 +\,\frac{3}{4}\,p_2^2+\,\frac{1}{2}\,p_4 \biggr) \notag\\
& \phantom{=} + \biggl( \frac{253}{120}\,p_1^5 + \,\frac{31}{12}\,p_1^3p_2
+\,\frac{2}{3}\,p_1^2p_3+\,\frac{29}{8}\,p_1p_2^2+\,\frac{3}{4}\,p_1p_4+\,\frac{2}{3}\,p_2p_3+\,\frac{3}{5}\,p_5 \biggr)+\cdots.
\end{align*}

\subsection{The Species $\mc{C}$ of Cographs}

The numbers of labeled and unlabeled cographs are given in~\cite[A006351, A000084]{sloan}. Lemma~\ref{lem-cograph} gives a way to compute the cycle index $Z_{\mc{C}}$ of cographs recursively:
\begin{align*}
Z_{\mc{C}} &= p_1+(p_1^2+p_2)+\biggl(\frac{4}{3}\,p_1^3+2p_1p_2+ \frac{2}{3}\,p_3\biggr) \notag\\
& \phantom{=}\ \, +\biggl(\frac{13}{6}\,p_1^4+2p_1^2p_2+ \frac{4}{3}\,p_1p_3+ p_2^2+p_4 \biggr) \notag\\
& \phantom{=}\ \, +\biggl(\frac{59}{15}\,p_1^5+\frac{8}{3}\,p_1^2p_3+5p_1p_2^2+\frac{26}{3}\,p_1^3p_2+ 2p_1 p_4 + \frac{4}{3}\,p_2p_3+\frac{2}{5}\,p_5\biggr)\notag\\
& \phantom{=} \ \, +\biggl( \frac{344}{45}\,p_1^6+  \frac{4}{5}\,p_1 p_5+  \frac{59}{3}\,p_1^4 p_2+\frac{52}{9}\,p_1^3 p_3+ 15p_1^2 p_2^2+    4p_1^2p_4+  \frac{16}{3}\,p_1p_2p_3  \biggr.\notag\\
& \phantom{=} \ \,  \biggl.  + 3p_2p_4 + p_6
\biggr) + \cdots.
\end{align*}

\subsection{The Species $\mc{B}$ of Bi-Point-Determining Graphs}

The numbers of labeled and unlabeled bi-point-determining graphs are given in~\cite[A129583, A129584]{sloan}. We obtain from Corollary~\ref{cor-enumbipd} the cycle index $Z_{\mc{B}}$ of bi-point-determining graphs from the cycle index $Z_{\mc{G}}$ of graphs.
\begin{align*}
Z_{\mc{B}}  &=  p_1 + \biggl( \frac{1}{2}\,p_1^4 + \frac{1}{2}\,p_2^2\biggr) +
 \biggl( \frac{13}{5}\,p_1^5+3p_1p_2^2+\frac{2}{5}\,p_5 \biggr) \notag \\
  & \phantom{=}\ + \biggl( \frac{96}{5}\, p_1^6 + 11 p_1^2p_2^2 + \frac{4}{5}\,p_1p_5
  + \frac{11}{3}\, p_2^3 + p_3^2 + \frac{1}{3}\, p_6  \biggr) +
  \cdots.
\end{align*}
Since the automorphism groups of bi-point-determining graphs contain no 2-, 3-, or 4-cycles, the cycle index  $Z_{\mc{B}}$ contains no terms of the form $p_1^n p_2$, $p_1^np_3$, or $p_1^n p_4$.

The molecular decomposition of $\mc{B}$ (see Figure~\ref{f-bipd_unlabel}) begins with
\begin{align*}
\mc{B} &= X + \mc{E}_2 \circ X^2 +( 5 X \cdot (\mc{E}_2 \circ X^2) + \mc{D}_5 ) + \cdots.
\end{align*}

The numbers of labeled and unlabeled connected bi-point-determining graphs are give in~\cite[A129585, A129586]{sloan}. We obtain the cycle index $Z_{\mc{B}^c}$ from Corollary~\ref{cor-connbipd}.
\begin{align*}
Z_{\mc{B}^c}&=p_1+\biggl(\frac{1}{2}\,p_1^4+\frac{1}{2}p_2^2\biggr)
+\biggl(\frac{21}{10}\,p_1^5+\frac{5}{2}\,p_1p_2^2+\frac{2}{5}p_5\biggr)\notag\\
&\phantom{=}\ +\biggl(\frac{17}{10}\,p_1^6+\frac{17}{2}\,p_1^2p_2^2
+\frac{2}{5}\,p_1p_5+\frac{11}{3}\,p_2^3+p_3^2\frac{1}{3}\,p_6\biggr)
+\cdots.
\end{align*}

\subsection{The Species $\mc{G}(X,Y)$ of bicolored Graphs}

Theorem~\ref{thm-2color} enables us to calculate the associated series of $\mc{G}[X,Y]$.
\begin{align}
\mc{G}(x,y)&=1+\,\frac{x}{1!}+\,\frac{x^2}{2!}+\,\frac{x^3}{3!}+\,\frac{y}{1!}+\,\frac{y^2}{2!}+\,\frac{y^3}{3!}+2\,\frac{x}{1!}\,\frac{y}{1!}+4\,\frac{x^2}{2!}\,\frac{y}{1!}+4\,\frac{x}{1!}\,\frac{y^2}{2!}+\cdots,\label{eqn-2color-label}\\
\widetilde{\mc{G}}(x,y)&=1+x+y+2xy+x^2+y^2+3x^2y+3xy^2+x^3+y^3+\cdots,\label{eqn-2color-unlabel}\\
Z_{\mc{G}(X,Y)}&=1+(p_1[\textbf{x}]+p_1[\textbf{y}])+\biggl(\,\frac{1}{2}\,p_1^2[\textbf{x}]+\,\frac{1}{2}\,p_2[\textbf{x}]+2p_1[\textbf{x}]p_1[\textbf{y}]+\,\frac{1}{2}\,p_2[\textbf{y}]+\,\frac{1}{2}\,p_1^2[\textbf{y}]\biggr)\notag\\
&\phantom{=}\ +\biggl(\,\frac{1}{6}\,p_1^3[\textbf{x}]+\,\frac{1}{2}\,p_1[\textbf{x}]p_2[\textbf{x}]+\,\frac{1}{3}\,p_3[\textbf{x}]+p_2[\textbf{x}]p_1[\textbf{y}]+2p_1^2[\textbf{x}]p_1[\textbf{y}]\biggr.\notag\\
&\phantom{=}\ \,\biggl.+2p_1[\textbf{x}]p_1^2[\textbf{y}]+p_1[\textbf{x}]p_2[\textbf{y}]+\,\frac{1}{3}\,p_3[\textbf{y}]+\,\frac{1}{2}\,p_1[\textbf{y}]p_2[\textbf{y}]+\,\frac{1}{6}\,p_1^3[\textbf{y}]\biggr)+\cdots.\notag
\end{align}

Equation~\eqref{eqn-2colorconn} enables us to compute the associated series of ${\mc{G}^c}[X,Y]$:
\begin{align}
{\mc{G}^c}(x,y)&=\,\frac{x}{1!}+\,\frac{y}{1!}+\,\frac{x}{1!}\,\frac{y}{1!}+\,\frac{x^2}{2!}\,\frac{y}{1!}+\,\frac{x}{1!}\,\frac{y^2}{2!}+\,\frac{x^3}{3!}\,\frac{y}{1!}+5\,\frac{x^2}{2!}\,\frac{y^2}{2!}+\,\frac{x}{1!}\,\frac{y^3}{3!}+\cdots,\notag\\
\widetilde{{\mc{G}^c}}(x,y)&=x+y+xy+xy^2+x^2y+x^3y+xy^3+2x^2y^2+x^4y+4x^3y^2+4x^2y^3\notag\\
&\phantom{=}\ \,+xy^4+\cdots,\label{eqn-2colorconn-unlabel}\\
Z_{{\mc{G}^c}(X,Y)}&=(p_1[\textbf{x}]+p_1[\textbf{y}])+p_1[\textbf{x}]p_1[\textbf{y}]+\biggl(\frac{1}{2}\,p_1^2[\textbf{x}]p_1[\textbf{y}]+\,\frac{1}{2}\,p_1[\textbf{x}]p_1^2[\textbf{y}]+\,\frac{1}{2}\,p_2[\textbf{x}]p_1[\textbf{y}]\biggr.\notag\\
&\phantom{=}\ \,+\biggl.\,\frac{1}{2}\,p_1[\textbf{x}]p_2[\textbf{y}]\biggr)+\cdots.\notag
\end{align}

The molecular decomposition of the $2$-sort species $\mc{G}(X,Y)$ begins with
\begin{align}
\mc{G}(X,Y)&=1+(X+Y)+\big[\mc{E}_2(X)+\mc{E}_2(Y)+2X\cdot Y\big]+\big[\mc{E}_3(X)+\mc{E}_3(Y)+X\cdot Y^2\notag\\
&\phantom{=}\ \,+X^2\cdot Y+2X\cdot \mc{E}_2(Y)+2\mc{E}_2(X)\cdot Y\big]+\cdots.\notag
\end{align}

The molecular decomposition of the $2$-sort species $\mc{G}^c(X,Y)$ begins with (see Figure~\ref{f-2color4_conn}):
\begin{align}
\mc{G}^c(X,Y) &= (X + Y) + X\cdot  Y + [X\cdot  \mc{E}_2(Y)+Y\cdot  \mc{E}_2(X)] \notag\\
&\phantom{=}\ \, + [X\cdot  \mc{E}_3(Y)+Y\cdot  \mc{E}_3(X)+X^2\cdot  Y^2 +\mc{E}_2(X)\cdot \mc{E}_2(Y)] +\cdots. \notag
\end{align}

\subsection{The Species $\mc{P}(X,Y)$ of bicolored point-determining graphs}

We obtain from Theorems~\ref{thm-2colorps} and~\ref{thm-2colorpd} the associated series for $\mc{P}^s(X, Y), \mc{P}^c(X, Y)$ and $\mc{P}(X, Y)$:
\begin{align}
Z_{\mc{P}^s(X,Y)}&=1+(p_1[\textbf{x}]+p_1[\textbf{y}])+(2p_1[\textbf{x}]p_1[\textbf{y}])+(p_1^2[\textbf{x}]p_1[\textbf{y}]+p_1[\textbf{x}]p_1^2[\textbf{y}])\notag\\
&\phantom{=}\ \,+\biggl(\frac{1}{2}\,p_2[\textbf{x}]p_2[\textbf{y}]+\,\frac{5}{2}\,p_1^2[\textbf{x}]p_1^2[\textbf{y}]\biggr)\notag\\
&\phantom{=}\ \,+(p_1p_2[\textbf{x}]p_2[\textbf{y}]+p_2[\textbf{x}]p_1p_2[\textbf{y}]+2p_1^3[\textbf{x}]p_1^2[\textbf{y}]+2p_1^2[\textbf{x}]p_1^3[\textbf{y}])+\cdots.\notag\\
Z_{\mc{P}^c(X,Y)}&=(p_1[\textbf{x}]+p_1[\textbf{y}])+(p_1[\textbf{x}]p_1[\textbf{y}])+(p_1^2[\textbf{x}]p_1^2[\textbf{y}])\notag\\
&\phantom{=}\ \,+\biggl(\frac{1}{2}\,p_1p_2[\textbf{x}]p_2[\textbf{y}]+\frac{1}{2}\,p_2[\textbf{x}]p_1p_2[\textbf{y}]+\frac{1}{2}\,p_1^3[\textbf{x}]p_1^2[\textbf{y}]+\frac{1}{2}\,p_1^2[\textbf{x}]p_1^3[\textbf{y}]\biggr)+\cdots.\notag
\end{align}
\begin{align}
\mc{P}(x,y)&=1+\,\frac{x}{1!}\,+\,\frac{y}{1!}\,+\,\frac{x}{1!}\,\frac{y}{1!}\,+2\,\frac{x}{1!}\,\frac{y^2}{2!}\,+2\,\frac{x^2}{2!}\,\frac{y}{1!}\,+6\,\frac{x^2}{2!}\,\frac{y^2}{2!}\,+24\,\frac{x^2}{2!}\,\frac{y^3}{3!}\,+24\,\frac{x^3}{3!}\,\frac{y^2}{2!}\,+\cdots,\notag\\
\widetilde{\mc{P}}(x,y)&=1+x+y+xy+x^2y+xy^2+2x^2y^2+3x^3y^2+3x^2y^3+\cdots,\notag\\
Z_{\mc{P}(X,Y)}&=1+(p_1[\textbf{x}]+p_1[\textbf{y}])+(p_1[\textbf{x}]p_1[\textbf{y}])+(p_1^2[\textbf{x}]p_1[\textbf{y}]+p_1[\textbf{x}]p_1^2[\textbf{y}])\notag\\
&\phantom{=}\ \,+\biggl(\frac{1}{2}\,p_2[\textbf{x}]p_2[\textbf{y}]+\,\frac{3}{2}\,p_1^2[\textbf{x}]p_1^2[\textbf{y}]\biggr)\notag\\
&\phantom{=}\ \,+(p_1p_2[\textbf{x}]p_2[\textbf{y}]+p_2[\textbf{x}]p_1p_2[\textbf{y}]+2p_1^3[\textbf{x}]p_1^2[\textbf{y}]+2p_1^2[\textbf{x}]p_1^3[\textbf{y}])+\cdots.\notag
\end{align}

The beginning terms of the molecular decomposition of $\mc{P}(X,Y)$ are (see Figure~\ref{f-2color5_pd}):
\begin{align}
\mc{P}(X,Y) &= 1+(X+Y)+ X\cdot  Y + (X^2\cdot  Y+X\cdot  Y^2) + \big[ \mc{E}_2(X)\cdot  \mc{E}_2(Y) + X^2 \cdot  Y^2 \big] \notag\\
& \phantom{=} \ \, + \big[(X+Y)\cdot  \mc{E}_2(X)\cdot  \mc{E}_2(Y) + (X+Y)\cdot  X^2 \cdot  Y^2 + X^3\cdot  Y^2 \big. \notag\\
& \phantom{=}\ \, \big. +X^2\cdot  Y^3 \big]+\cdots. \notag
\end{align}

\bibliographystyle{gessel} 
\bibliography{pointd}
\end{document}